\documentclass[preprint]{elsarticle}
\usepackage{amsmath, latexsym, amsfonts, amssymb, amsthm, amscd, mathrsfs}

\usepackage{a4wide}
\usepackage[utf8]{inputenc}
\usepackage[english]{babel}

\usepackage{graphicx}
\usepackage{psfrag}
\usepackage{xcolor,rotating}
\usepackage{multirow, multicol}
\usepackage{arydshln}
\usepackage{blkarray}
\usepackage[lofdepth,lotdepth]{subfig}
\usepackage{enumerate}
\usepackage{url}
\usepackage{dsfont}
\usepackage{MnSymbol}

\newcommand{\bC}{\ensuremath{\mathbf{C}}}

\newcommand{\bE}{\ensuremath{\mathbf{E}}}

\newcommand{\bH}{\ensuremath{\mathbf{H}}}

\newcommand{\bL}{\ensuremath{\mathbf{L}}}

\newcommand{\bN}{\ensuremath{\mathbf{N}}}

\newcommand{\bR}{\ensuremath{\mathbf{R}}}
\newcommand{\bS}{\ensuremath{\mathbf{S}}}

\newcommand{\bZ}{\ensuremath{\mathbf{Z}}}

\newcommand{\bbC}{\ensuremath{\mathbb{C}}}

\newcommand{\bbR}{\ensuremath{\mathbb{R}}}

\newcommand{\mA}{\ensuremath{\mathcal{A}}}
\newcommand{\mB}{\ensuremath{\mathcal{B}}}
\newcommand{\mC}{\ensuremath{\mathcal{C}}}
\newcommand{\mD}{\ensuremath{\mathcal{D}}}
\newcommand{\mE}{\ensuremath{\mathcal{E}}}
\newcommand{\mF}{\ensuremath{\mathcal{F}}}
\newcommand{\mG}{\ensuremath{\mathcal{G}}}
\newcommand{\mH}{\ensuremath{\mathcal{H}}}

\newcommand{\mL}{\ensuremath{\mathcal{L}}}
\newcommand{\mM}{\ensuremath{\mathcal{M}}}
\newcommand{\mN}{\ensuremath{\mathcal{N}}}
\newcommand{\mO}{\ensuremath{\mathcal{O}}}

\newcommand{\mS}{\ensuremath{\mathcal{S}}}

\newcommand{\mU}{\ensuremath{\mathcal{U}}}
\newcommand{\mV}{\ensuremath{\mathcal{V}}}

\newcommand{\tA}{\ensuremath{\tilde{A}}}

\newcommand{\tT}{\ensuremath{\tilde{T}}}

\newcommand{\tmD}{\ensuremath{\widetilde{\mathcal{D}}}}

\def\lamb{\lambda}

\def\eps{\varepsilon}

\def\om{\omega}

\def\vphi{\varphi}

\def\ind#1{\mathds{1}_{#1}}
\newcommand{\tta}{\ensuremath{\theta}}
\def\N#1{\left\|\,#1\,\right\|}

\def\Ninf#1{\left\|\,#1\,\right\|_{\infty}}

\newcommand{\Nd}[1]{\left\|\,#1\,\right\|_{2}}

\newcommand{\Ndmu}[1]{\left\|\,#1\,\right\|_{\mu, 2}}
\newcommand{\NsQmu}[1]{\left\|\,#1\,\right\|_{\mu, -1, q}}

\newcommand{\Nsq}[1]{\left\|\,#1\,\right\|_{-1, q_0}}

\newcommand{\NdPhi}[1]{\left\|\,#1\,\right\|_{2, w}}
\newcommand{\Ndk}[1]{\left\|\,#1\,\right\|_{2, k}}
\newcommand{\NsPhi}[1]{\left\|\,#1\,\right\|_{-1, w}}

\newcommand{\NH}[1]{\left\|\,#1\,\right\|_{\bH_{\mu}}}

\newcommand{\NHw}[1]{\left\|\,#1\,\right\|_{\Hw}}

\newcommand{\cro}[2]{\ensuremath{\left\langle#1\, ,\, #2\right\rangle}}

\newcommand{\crosqqmu}[2]{\ensuremath{\left\langle#1\, ,\, #2\right\rangle_{\mu, -1, q}}}
\newcommand{\crodmu}[2]{\ensuremath{\left\langle#1\, ,\, #2\right\rangle_{\mu, 2}}}

\newcommand{\crosPhi}[2]{\ensuremath{\left\langle#1\, ,\, #2\right\rangle_{-1, w}}}

\newcommand{\crodPhi}[2]{\ensuremath{\left\langle#1\, ,\, #2\right\rangle_{2, w}}}

\newcommand{\croH}[2]{\ensuremath{\left\langle#1\, ,\, #2\right\rangle_{\bH_{\mu}}}}
\newcommand{\hp}{h_{+}}
\newcommand{\hm}{h_{-}}

\newcommand{\HsPhi}{\ensuremath{\bH^{-1}_{w}}}

\newcommand{\Hsq}{\ensuremath{\bH^{-1}_{{q_0}}}}

\newcommand{\Hsmu}{\ensuremath{\bH^{-1}_{{\mu}}}}

\newcommand{\Hw}{\ensuremath{H_w}}
\newcommand{\Hz}{\ensuremath{\bH_0}}

\newcommand{\Ldk}{\ensuremath{\bL^2_{k}}}
\newcommand{\LdPhi}{\ensuremath{\bL^2_{w}}}
\newcommand{\Ldkmu}{\ensuremath{\bL^2_{\mu, k}}}

\newcommand{\Ldmu}{\ensuremath{\bL^2_{\mu}}}
\newcommand{\Ld}{\ensuremath{\bL^2}}

\DeclareMathOperator{\Cov}{Cov}
\DeclareMathOperator{\Span}{Span}

\def\ens#1#2{\left\{#1;\,#2\right\}}
\def\Som#1#2#3{\sum_{#1=#2}^{#3}}
\DeclareMathOperator*{\Supp}{Supp}
\newcommand{\intS}{\int_{\Sun}}
\newcommand{\intR}{\int_{\bR}}
\newcommand{\intSR}{\int_{\Sun\times\bR}}
\newcommand{\Sun}{\ensuremath{\mathbf{S}}}

\DeclareMathSymbol{\leqslant}{\mathalpha}{AMSa}{"36} 
\DeclareMathSymbol{\geqslant}{\mathalpha}{AMSa}{"3E} 
\DeclareMathSymbol{\eset}{\mathalpha}{AMSb}{"3F}     
\renewcommand{\leq}{\;\leqslant\;}                   
\renewcommand{\geq}{\;\geqslant\;}                   
\newcommand{\dd}{\,\text{\rm d}} 

\newcommand{\wrt}{w.r.t.\ }
\newcommand{\iid}{i.i.d.\ }

\newcommand{\half}{\ensuremath{\frac{1}{2}}}
\newcommand{\gapA}{\ensuremath{\lambda_K(A)}}
\newcommand{\gapAf}{\ensuremath{\lambda_K(\tA_1)}}
\newcommand{\gapAs}{\ensuremath{\lambda_K(\tA_2)}}
\newcommand{\gapz}{\ensuremath{\lambda_K(L_{q_0})}}
\newcommand{\gapL}{\ensuremath{\lambda_K(L)}}
\newcommand{\Pneg}{P_{\mbox{\tiny $\!\!<\!\!0$}\,}}
\newcommand{\Pzero}{P_{0}\,}
\newcommand{\Gneg}{G_{\mbox{\tiny $\!<\!\!0$}\,}}
\newcommand{\Fneg}{F_{\mbox{\tiny $\!\!<\!\!0$}\,}}
\newcommand{\vect}[2]{\begin{pmatrix} #1 \\ #2 \end{pmatrix}}
\newcommand{\svect}[2]{\left(\begin{smallmatrix} #1 \\ #2 \end{smallmatrix}\right)}

\numberwithin{equation}{section}

\newtheorem{theo}{Theorem}[section]
\newtheorem{lem}[theo]{Lemma}
\newtheorem{prop}[theo]{Proposition}

\newtheorem{rem}[theo]{Remark}

\begin{document}
\title{Large time asymptotics for the fluctuation SPDE in the Kuramoto
synchronization model}
\author[el]{Eric Lu{\c c}on}
\ead{eric.lucon@parisdescartes.fr}
\date{\today}
\address[el]{Universit{\'e} Pierre et Marie Curie (Paris 6) and Laboratoire de Probabilit{\'e}s et Mod\`eles Al\'eatoires (CNRS),
U.F.R. Math\'ematiques, Case 188, 4 place Jussieu, 75252 Paris cedex 05, France\\[10pt]
(Current address: Laboratoire MAP5 - Universit\'e Ren\'e Descartes - Paris 5, UFR Math\'ematiques et Informatique, 45 rue des Saints-P\`eres, 75270 Paris cedex 06, France)}



\begin{abstract}
We investigate the long-time asymptotics of the fluctuation SPDE in the Kuramoto synchronization model. We establish the linear behavior for large time and weak disorder of the quenched limit fluctuations of the empirical measure of the particles around its
McKean-Vlasov limit. This is carried out through a spectral analysis of the underlying unbounded evolution operator, using arguments of perturbation of self-adjoint operators and analytic semigroups. We state in particular a Jordan decomposition of the evolution operator which is the key point in order to show that the fluctuations of the disordered Kuramoto model are not self-averaging.
\end{abstract}
\maketitle

\section{Introduction}
\label{sec:introstabfluct}

\subsection{Synchronization of heterogeneous oscillators}
Collective behavior of oscillators and synchronization phenomenon are the subject of a vast literature either in biological (neuronal models, collective behavior of insects, cells, etc.) or in physical contexts (see~\cite{Kuramoto1984,Strogatz1991} and references therein). While a precise description of each of the different instances in which synchronization emerges demands specific, possibly very complex models, the Kuramoto model~\cite{Acebr'on2005} has emerged as capturing some of the fundamental aspects of synchronization.

The \emph{disordered Kuramoto} model concerns a family of heterogeneous oscillators (or rotators) on the circle $\Sun:=\bR/2\pi\bZ$ in a noisy mean-field interaction (that is the dynamics is perturbed by thermal noise). Each rotator obeys to its own natural frequency which may differ from one rotator to another. Those frequencies are chosen at random and independently for each rotator according to a probability distribution $\mu$ on $\bR$; hence, this supplementary source of randomness will be considered as a \emph{disorder}. 

One of the main characteristics of the Kuramoto model is that it presents a phase transition, as the coupling strength between rotators increases, from an incoherent state where the rotators are not synchronized to a synchronous one where the phases of the rotators concentrate around a common value (for a review on the subject, see \cite{Acebr'on2005}). In this context, the question of how the random frequencies influence synchronization has been raised by many authors, not only in the Kuramoto model (\cite{Strogatz1991}) but also for more general models of weakly interacting diffusions (e.g. neuronal models, see \cite{22657695} and references therein).

\subsubsection{The continuous model}
The disordered Kuramoto model~\cite{Kuramoto1984, Acebr'on2005}, in the limit of an infinite population of rotators, is described by the
following nonlinear Fokker-Planck equation (or \emph{McKean-Vlasov equation}): 
\begin{equation}
\label{eq:fokker planck intro limfluct}
 \partial_t q_t(\tta, \om) = \half \partial_\tta^2 q_t(\tta, \om) - \partial_{\tta} \Big( q_t(\tta, \om)\left(
\langle J \ast q_t\rangle_\mu(\tta) + \om\right) \Big),\quad t>0,\ \tta\in\Sun,\ \om\in\Supp(\mu),
\end{equation}
with periodic boundary conditions and initial condition given by 
\begin{equation}
\label{eq:initial condition fokker planck}
\forall \om\in\Supp(\mu),\quad q_t(\tta, \om)\dd\tta\ \substack{\mbox{}\\\xrightarrow{\hspace*{15pt}}\\t\searrow0}\ \gamma(\dd\tta),
\end{equation}
for some probability law $\gamma$ on the circle $\Sun$. Here, \begin{equation}
\label{eq:J}
\langle J\ast q_t\rangle_\mu(\tta) = -K\int_{\bR}\int_{\Sun}{\sin(\vphi) q_t(\tta-\vphi, \om)\dd\vphi\mu(\dd\om)},
\end{equation}
stands for the convolution of $J(\cdot):=-K\sin(\cdot)$ with $q_t$, averaged with respect to $\om$ and $K$ is the positive coupling strength between rotators. Note that we are looking for solutions $(t, \tta, \om)\mapsto q_t(\tta,\om)$ that are probability densities for all fixed $t$ and $\om$: $q_t(\cdot, \cdot)\geq 0$ for all $t>0$ with $\intS q_t(\tta, \om)\dd\tta=1$ for all $\om\in\Supp(\mu)$. 

The physical interpretation of \eqref{eq:fokker planck intro limfluct}--\eqref{eq:initial condition fokker planck} is the following: $\tta\in\Sun$ is the phase of the rotators, $\gamma$ is their initial distribution on $\Sun$, $\mu$ is the probability distribution of the frequencies, and $q_{t}(\tta, \om)$ is the density of rotators with phase $\tta$ and frequency $\om$ at time $t>0$.

\medskip
Uniqueness of a solution to \eqref{eq:fokker planck intro limfluct}--\eqref{eq:initial condition fokker planck} follows from standard arguments concerning fundamental solutions
of parabolic equations (\cite{Aronson1971, Friedman1964}) and has been rigorously established in~\cite[\S~A]{GLP2011}. Another proof of
uniqueness can be found in~\cite{daiPra96} on the basis of heat kernel estimates under regularity assumptions on the initial condition.

\subsubsection{The microscopic model}
\label{subsubsec:microscopic model}
Existence of a solution to \eqref{eq:fokker planck intro limfluct} can be seen as a consequence of the following probabilistic interpretation: for all $N\geq 1$, consider the following system of $N$ stochastic differential equations in a mean-field interaction
\begin{equation}
\label{eq:Kmod}
\dd\tta_{j,t}\, =\, \om_j \dd t - \frac KN \sum_{i=1}^N \sin(\tta_{j, t}-\tta_{i, t}) \dd t + \dd B_{j, t},\quad j=1, \ldots, N,\ t\geq0,
\end{equation}
where at time $t=0$, the rotators $\tta_{j, 0}$ are \iid with law $\gamma$, $(\om_{i})_{i=1, \ldots, N}$ are \iid with law $\mu$ and $\{B_j\}_{j=1, \ldots,  N}$ are $N$ standard independent Brownian motions. Evolution \eqref{eq:fokker planck intro limfluct} appears naturally as the large $N$-limit of the system \eqref{eq:Kmod} in the following way: if one defines the empirical measure $\nu_N$ of both rotators and frequencies as
\begin{equation}
 \label{eq:empirical measure intro limfluct}
t\mapsto\nu_{N,t}:= \frac1N \Som{j}{1}{N}{\delta_{(\tta_{j,t}, \om_j)}} \in\mC([0, +\infty),\mM_1(\Sun\times\bR)),
\end{equation}
where $\delta_{(\tta, \om)}$ is the Dirac measure in $(\tta, \om)$ and $\mM_1(\Sun\times\bR)$ the set of probability measures on
$\Sun\times\bR$, it can be shown (see~\cite{daiPra96, Lucon2011}) that, under mild hypotheses, the sequence $\nu_N$ converges as $N$ goes to $+\infty$ in law (as a process), to the deterministic limit $t\mapsto\nu_t$ such that 
\begin{equation}
\label{eq:nu limit}
 \left\{\begin{array}{rcl}
\nu_0(\dd\tta, \dd\om)&=&\gamma(\dd\tta)\otimes\mu(\dd\om)\\
\nu_t(\dd\tta, \dd\om)&=&q_t(\tta, \om)\dd\tta\mu(\dd\om), \quad t>0, 
        \end{array}\right.
\end{equation}
where $q_t$ is solution of the McKean-Vlasov equation \eqref{eq:fokker planck intro limfluct}. 
\begin{rem}
\label{rem:selfaveragingLLN}\rm
Due to the mean-field character of \eqref{eq:Kmod}, there is a \emph{self-averaging phenomenon} (see~\cite[Th.~2.5]{Lucon2011}): the above convergence
is true \emph{for almost every choice of the frequencies $(\om_j)_{j\geq1}$} and the limit $\nu$ does not depend on this initial choice. 
\end{rem}
This law of large numbers is a disordered generalization of known results about mean-field interacting diffusions (see e.g.~\cite{Gartner, Jourdain1998, McKean1967, Oelsch1984} for similar situations without disorder). Note also that this convergence is also valid for more general models (see e.g. the recent work on the Winfree model~\cite{MR3032864} or FitzHugh-Nagumo and Hodgkin-Huxley models of neuronal oscillators \cite{22657695}).

\subsection{The fluctuation SPDE}

In this paper, we investigate the asymptotic behavior as $t\to+\infty$ of the following stochastic partial differential equation (SPDE):
\begin{equation}
  \label{eq:SPDE fluctuations limfluct}
\eta_{t} = \eta_0 + \int_{0}^{t}{L_{q_s}\eta_{s} \dd s} + W_{t},
\end{equation}
where $L_{q_s}$ is the linearized operator around the solution $t\mapsto q_t$ of nonlinear evolution \eqref{eq:fokker planck intro limfluct}:
\begin{equation}
\label{eq:L linear qt intro limfluct}
L_{q_t}\vphi(\tta, \om):= \half \partial_\tta^2 \vphi(\tta, \om) - \partial_{\tta} \left(\vphi(\tta, \om)\left(
\langle J \ast q_t\rangle_\mu(\tta) + \om\right) + q_t(\tta, \om) \langle J \ast \vphi\rangle_\mu(\tta)\right),
\end{equation}
where $\vphi$ is a regular function, $W$ is a Gaussian process, indexed by functions $\vphi(\tta, \om)$ such that $\partial_\tta\vphi(\cdot, \om)\in\bL^2(\Sun)$ for all
$\om\in\Supp(\mu)$, with covariance
\begin{equation}
 \label{eq:W process intro limfluct}
\forall\vphi_1,\vphi_2,\ \bE(W_{t}(\vphi_{1})W_{s}(\vphi_{2})) = \int_{0}^{s\wedge t}\intSR \partial_\tta\vphi_{1}(\tta,
\om)\partial_\tta\vphi_{2}(\tta,
\om)q_u(\tta,\om)\dd\tta\mu(\dd\om)\dd u,
\end{equation}
and where the initial condition $\eta_0$ is independent of $W$.
\subsubsection{The SPDE \eqref{eq:SPDE fluctuations limfluct} as the limit of the fluctuation process}
The SPDE \eqref{eq:SPDE fluctuations limfluct} is the natural limit object in the Central Limit
Theorem associated to the convergence as $N\to+\infty$ of the empirical measure $\nu_N$ \eqref{eq:empirical measure intro limfluct} towards its limit $\nu$
\eqref{eq:nu limit}. Namely, the object of a previous work~\cite[Th.~2.10]{Lucon2011} was to prove that the fluctuation process
\begin{equation}
 \label{eq:fluctuations process intro limfluct}
t\geq0\mapsto \eta_{N, t}:= \sqrt{N}\left( \nu_{N, t} - \nu_{t}\right),\ N\geq1,
\end{equation}
converges as $N\to\infty$, in a weak sense, in an appropriate space of distributions on $\Sun\times\bR$, to the solution $\eta$ of \eqref{eq:SPDE fluctuations limfluct}. 

Similar fluctuation results for interacting diffusions had already been considered in the literature (\cite{Fernandez1997, Jourdain1998}). The particularity of the above result is that it is a \emph{quenched} notion of fluctuation, which still keeps track of the influence of the disorder $(\om_{1}, \ldots, \om_{N})$ as $N\to\infty$. The
precise notion of convergence used in~\cite{Lucon2011} is not really relevant for the purpose of this paper; more details can be found in~\cite[Th.~2.10]{Lucon2011}. What we only need to retain here is that the limit $\eta= (\eta^{\om})_{\om}$ captures the dependence in the disorder through its mean-value: there exists a Gaussian process $\om\mapsto C(\om)$
with covariance
\begin{equation}
\label{eq:covariance C}
\forall \vphi_1, \vphi_2:\Sun\times\bR\to\bR,\ \Gamma_C(\vphi_1, \vphi_2)= \Cov_\mu\left( \intS \vphi_1(\cdot, \om)\dd\gamma, \intS
\vphi_2(\cdot, \om)\dd\gamma\right)
\end{equation}
such that for fixed $\om$, the initial condition $\eta_0^\om$ in \eqref{eq:SPDE fluctuations limfluct} may be written as 
\begin{equation}
\label{eq:struct etazero}
\eta_0^\om= X + C(\om), 
\end{equation}
where $X$ is an explicit centered Gaussian process. The mean-value $C(\cdot)$ has an interpretation in terms of the microscopic system \eqref{eq:Kmod}: it models in law the asymmetry in the initial choice of the frequencies $(\om_{1}, \ldots, \om_{N})$ as $N\to\infty$ (see \S~\ref{subsec:asymdisor} for further details).

\subsubsection{Finite size effects in the Kuramoto model: non self-averaging phenomenon}
\label{subsubsec:finitesize}
The motivation of this work is to study the influence of a typical realization of the frequencies $(\om_j)_{j\geq1}$ (\emph{quenched model}) on the behavior of \eqref{eq:Kmod} for large but finite $N$. Indeed, (as shown numerically in~\cite{Balmforth2000}), at the level of the microscopic system \eqref{eq:Kmod}, fluctuations of the frequencies $(\om_{i})_{i=1, \ldots, N}$ compete with the fluctuations of the thermal noise and make the whole system rotate: even in the simple case of $\mu=\frac12\left( \delta_{-1} + \delta_1\right)$, fluctuations in a finite sample $(\om_1, \ldots, \om_N)\in \{\pm1\}^N$ may lead to a majority of $+1$ with respect to $-1$, so that the rotators with positive frequency induce a global rotation of the whole system in the direction of the majority. Direction and speed of rotation depend on this initial random choice of the disorder (Fig.~\ref{fig:evoldens intro} and~\ref{subfig:fluctpsiN intro}). This can be noticed by computing the order parameters $\left(r_{N, t}, \psi_{N, t}\right)$ (recall~\eqref{eq:empirical measure intro limfluct}): 
\begin{equation}
\label{eq:intro order parameters N}
r_{N, t} e^{i\psi_{N, t}} = \frac{1}{N} \Som{j}{1}{N}{e^{i \tta_{j, t}}} = \intSR e^{i\tta}\dd\nu_{N, t}.(\tta, \om),\quad N\geq1,\ t\geq0,\end{equation}
Here $r_{N, t}\in [-1, 1]$ gives a notion of synchronization of the system (e.g. $r_{N,t}=1$ if the  oscillators $\tta_{j, t}$ are all equal) and $\psi_{N, t}$ captures the position of the center of synchronization (see Figure~\ref{fig:evoldens intro}). One can see on Figure~\ref{subfig:fluctpsiN intro} that $t\mapsto\psi_{N, t}$ has an approximately linear behavior whose slope depends on
the choice of the disorder. Note that this disorder-induced phenomenon does not happen at the level of the nonlinear Fokker-Planck equation~\eqref{eq:fokker planck intro limfluct}, but only at the level of fluctuations~\eqref{eq:fluctuations process intro limfluct} (the speed of rotation is of order $N^{-1/2}$).
\begin{figure}[!ht]
\centering
\includegraphics[width=0.8\textwidth]{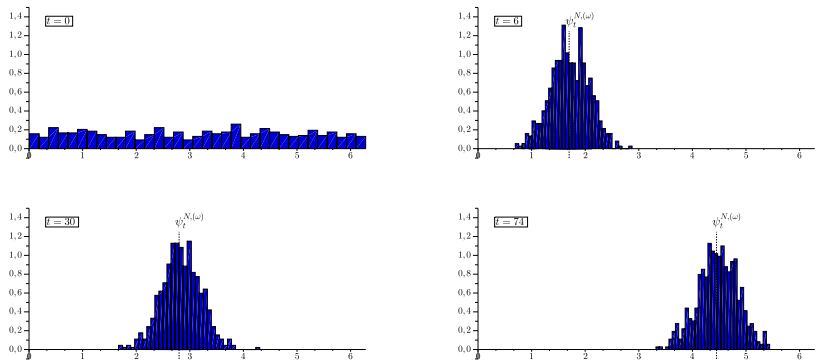}%
\caption{Evolution of the marginal on $\Sun$ of $\nu_N$ ($N=600$, $\mu=\frac{1}{2}(\delta_{-1}+ \delta_{1})$,
$K=6$). The rotators are initially independent and uniformly distributed on $[0, 2\pi]$ and independent of the disorder. First the dynamics leads to synchronization ($t=6$) to a profile close to a nontrivial stationary solution of~\eqref{eq:fokker planck intro limfluct}. Secondly, we observe that the center $\psi_{N, t}$ of this density moves to the
right with an approximately constant speed; this speed of fluctuation turns out to be sample-dependent (see Fig.~\ref{subfig:fluctpsiN intro}).}
\label{fig:evoldens intro}%
\end{figure}
\begin{figure}[!ht]
\centering
\subfloat[Trajectories of the center of synchronization $\psi_N$ for different
realizations of the disorder ($\mu=\frac{1}{2}(\delta_{-0.5}+ \delta_{0.5})$, $K=4$, $N=400$).]{\includegraphics[width=6.4cm]{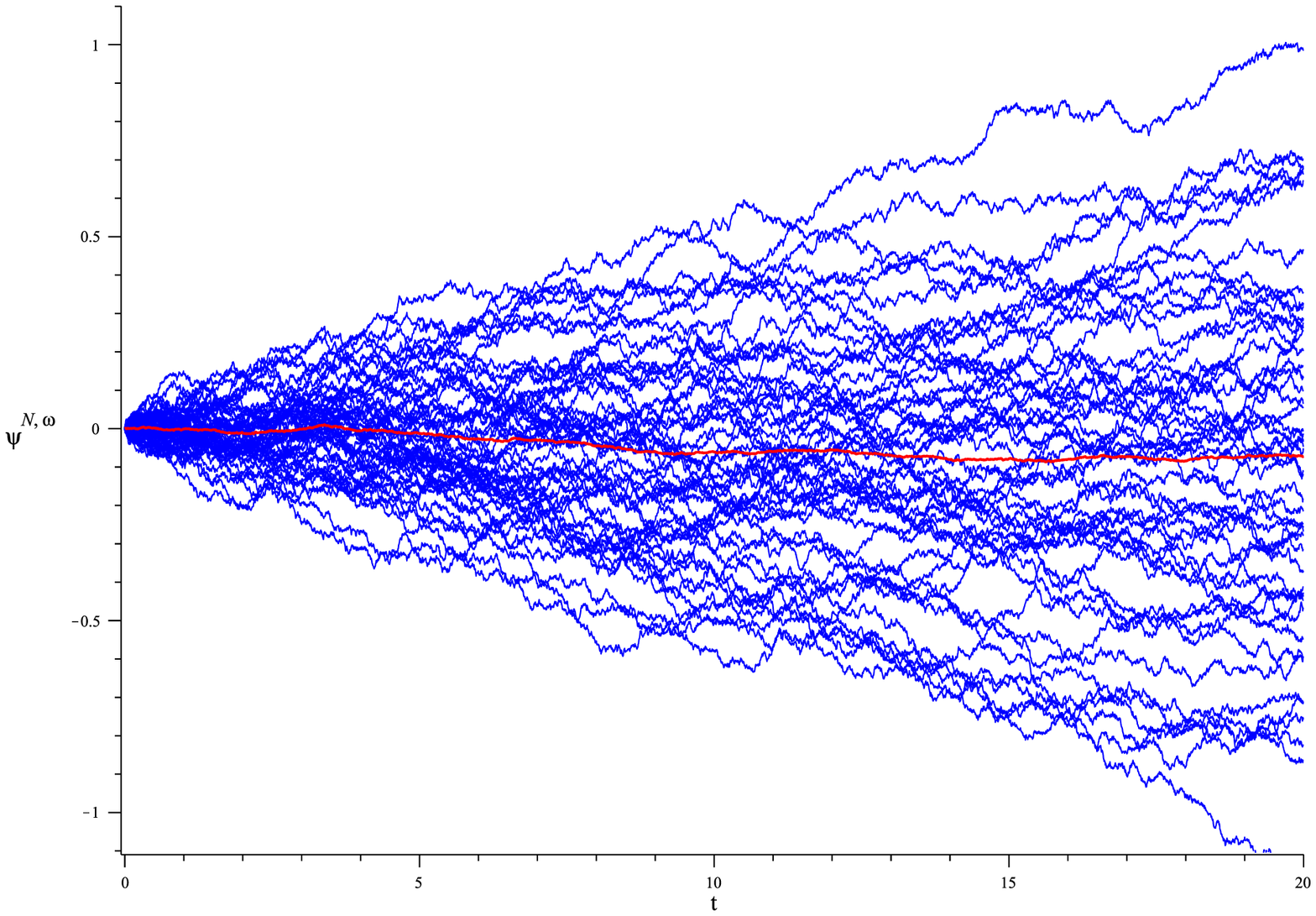}
\label{subfig:fluctpsiN intro}}\quad \subfloat[Trajectories of the process $\eta_t(\sin)$, for different
realizations of the mean-value $C$]{\includegraphics[width=7.7cm]{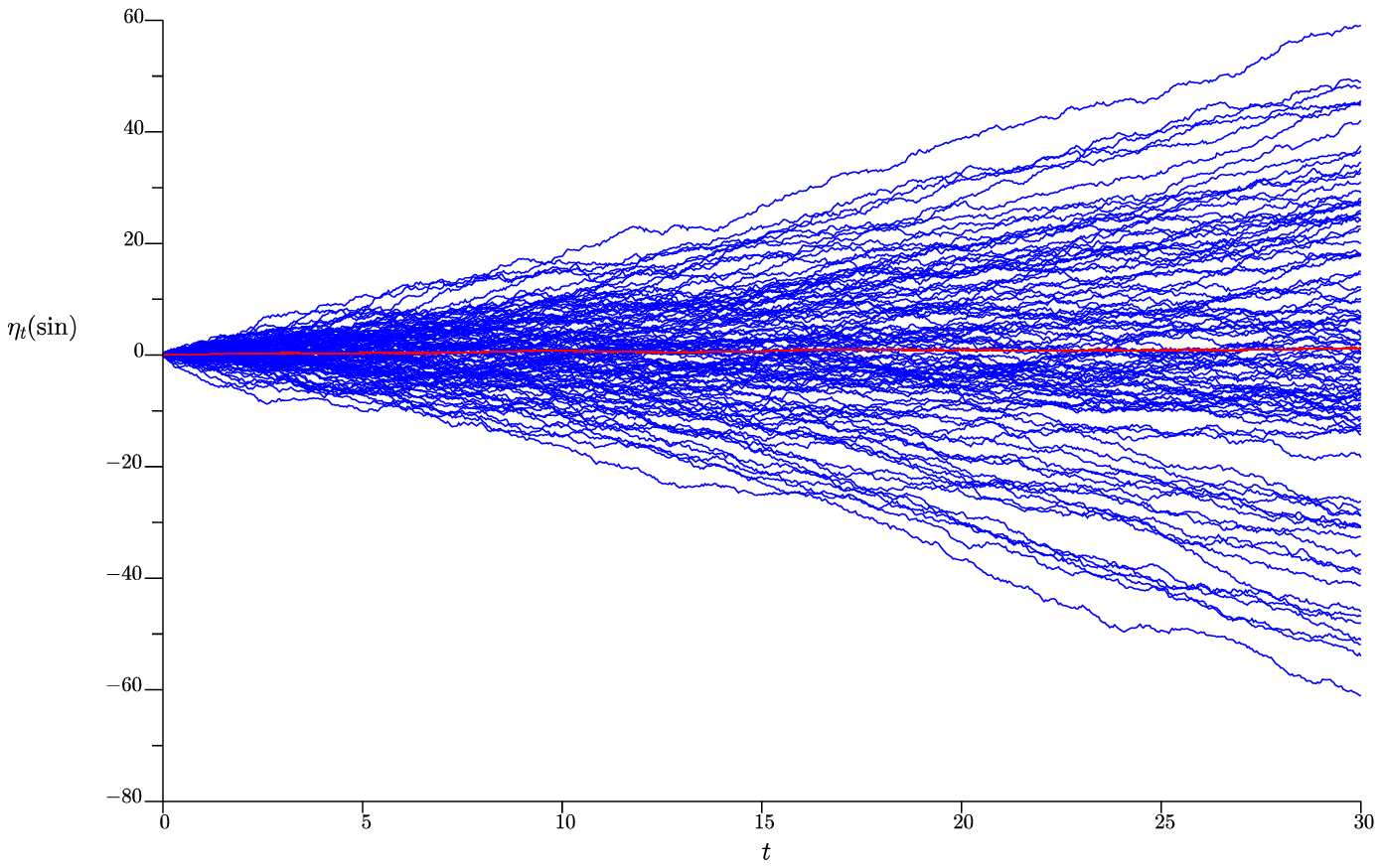}
\label{subfig:evoleta limfluct}}
\caption{Non self-averaging phenomenon in the Kuramoto model: in Fig.~\ref{subfig:fluctpsiN intro}, direction and speed of $\psi_{N}$ depend on the choice of the initial $N$-sample of the frequencies. Moreover, these simulations are compatible with speeds of order $N^{-1/2}$. In Fig.~\ref{subfig:evoleta limfluct}, trajectories of the fluctuation process $\eta(\sin)$ are sample-dependent and compatible with Fig~\ref{subfig:fluctpsiN intro}. }%
\label{fig:nonself}%
\end{figure}

\subsubsection{Long-time asymptotics of the fluctuation process}

What makes evolution \eqref{eq:SPDE fluctuations limfluct} relevant here is that its solution $\eta$ still captures this disorder-dependent rotation: at least
numerically, one observes trajectories of the process $\eta$ that are compatible with the ones observed for the finite-size system
\eqref{eq:Kmod} (see Figure~\ref{fig:nonself}).

Hence, a way to understand the phenomenon described in \S~\ref{subsubsec:finitesize} is to analyze the dependence of the fluctuation process $\eta$ \eqref{eq:SPDE fluctuations limfluct} in its mean-value $C$ (which, as we said, captures the initial asymmetry of the disorder). The key point of this paper is to understand how different initial conditions in evolution \eqref{eq:SPDE fluctuations limfluct} may lead to distinct approximately linear trajectories of the fluctuation process, as in Figure~\ref{subfig:evoleta limfluct}. 

Namely, in Theorem~\ref{theo:long time evolution eta limfluct}, we prove the following convergence for the solution $\eta$ of \eqref{eq:SPDE fluctuations limfluct}, in an appropriate space of distribution: for fixed $\om$
\begin{equation}
\label{eq:conv intro}
 \frac{\eta_t^{\om}}{t} \substack{\text{in law}\\\xrightarrow{\hspace*{30pt}}\\t\to+\infty} V(\om),
\end{equation}
where the speed $V(\om)$ (which depends on the initial condition $C(\om)$) has an explicit nontrivial law. This result relies on a detailed spectral analysis of the unbounded evolution operator $L_q$ defined in \eqref{eq:L linear qt intro limfluct}, using arguments from perturbation theory of self-adjoint operators (\cite{Kato1995}) and of analytic semigroups (\cite{Lunardi1995, Pazy1983}) and usual techniques about SPDEs in Hilbert spaces (\cite{DaPrato1992}). The main ingredient for this result consists in proving the existence of a nontrivial Jordan block for the eigenvalue $0$ for the operator $L_{q}$, relying on \emph{a priori} estimates on the Dirichlet form associated to $L_{q}$ and an extension of Lax-Milgram Theorem.

\subsection{Conclusion and perspectives}
The main conclusion of this work is that the Kuramoto model is not self-averaging at the level of fluctuations: the dynamics of the quenched fluctuations of \eqref{eq:Kmod} are still disorder-dependent, contrary to the dynamics of the nonlinear Fokker-Planck equation \eqref{eq:fokker planck intro limfluct}. However, in order to derive rigorously the exact speed of the rotation of synchronized solutions described in Figure~\ref{fig:nonself}, it would be necessary to study \eqref{eq:Kmod} on larger time scales (e.g. time scales of order $N$ as in \cite{2012arXiv1209.4537B}). This has not been carried out here and would be a natural perspective for this work.

The notion of self-averaging (or its absence) is crucial in many disordered models of statistical physics and is deeply related to the influence of the disorder on the phase transition in such systems (see e.g. \cite{Pastur:1991fk,ECP1365,MR904135} and references therein). We could not find any previous reference in the literature concerning non self-averaging for models of disordered interacting diffusions. 

One difficulty is that, although both law of large numbers $\nu_{N}\to\nu$ and central limit theorem $\sqrt{N}(\nu_{N}-\nu)\to \eta$ are valid in a rather general setting (see \cite{Fernandez1997, Lucon2011, 1301.6521}), investigating the long-time behavior of the limiting objects $\nu$ and $\eta$ is often very difficult (even well-posedness of the nonlinear Fokker-Planck equation is sometimes problematic, see \cite{Caceres:2011kx, 2012arXiv1211.0299D} for similar models of integrate-and-fire neurons). In that sense, one of the reasons for the popularity of the Kuramoto model is that the stationary solutions of the nonlinear Fokker-Planck equation \eqref{eq:fokker planck intro limfluct} are explicitly computable (see \S~\ref{subsec:mckean} below). Progress has recently been made in the stability analysis of its synchronized stationary solutions (\cite{GLP2011, GPPP2011}). A key point in this analysis is that the Kuramoto model without disorder is \emph{reversible} (\cite{BGP}), whereas reversibility is lost for many interesting neuronal models (e.g. FitzHugh-Nagumo~\cite{0951-7715-25-8-2303}). In that sense, it is remarkable that a similar stability analysis could be performed on the Winfree model of pulse oscillators in the recent work \cite{MR3032864}.

A second difficulty is that one needs to be in a \emph{quenched} set-up in order to see such a non self-averaging phenomenon: the \emph{averaged} Kuramoto model is indeed self-averaging at the level of fluctuations (see~\cite{Lucon2011}).

This work addresses the behavior as $t\to\infty$ of the fluctuation SPDE \eqref{eq:SPDE fluctuations limfluct}. It would be hopeless to review the vast literature (since~\cite{DaPrato1992, MR876085}) on long-time behavior of SPDEs (existence of invariant measures or random attractors have been studied for many models e.g. \cite{doi:10.1080/03605300500357998, MR2922844,MR1427258}). In our framework, the main difficulty of the long-time analysis of fluctuation for interacting diffusions (see e.g.~\cite{MR2912503}) lies in the fact that the dynamics of such systems is deeply related to the linear stability of their equilibria, which is, as we said, often hard to characterize and establish.

Concerning possible generalizations of this work, the results presented here should certainly be applicable to other disordered models of diffusions, provided sufficient information is known about characterization and linear stability of stationary states. In view of the recent work \cite{MR3032864}, the issue of wether or not similar non self-averaging results hold for the Winfree model is an intriguing question and would require further analysis.

\subsection{Organization of the paper}
The paper is organized as follows: in Section~\ref{sec:linstab}, we precise the main set-up for the study of the SPDE \eqref{eq:SPDE fluctuations limfluct} and state the main results. In
particular, Theorem~\ref{th:existence jrodan block intro limfluct} and Theorem~\ref{theo:spectreLq} deal with the spectral properties of the evolution operator $L_{q_t}$ at least when the disorder is small. Secondly, we state the main result of this paper: Theorem~\ref{theo:long time evolution eta limfluct} establishes the linear asymptotics of the fluctuation process solution of \eqref{eq:SPDE fluctuations limfluct}. Section~\ref{sec:jordan block limfluct} is devoted to prove Theorem~\ref{th:existence jrodan block intro limfluct}. In Section~\ref{sec:global localization spectrum L limfluct} we prove Theorem~\ref{theo:spectreLq}, whereas the main result of the paper, Theorem~\ref{theo:long time evolution eta limfluct} is proved in Section~\ref{sec:limit behavior etat limfluct}.
\section{Main definitions and results}
\label{sec:linstab}
\subsection{Long-time analysis of the McKean-Vlasov equation}
\label{subsec:mckean}
Before going into the details of the analysis of the SPDE \eqref{eq:SPDE fluctuations limfluct}, let us recall some results concerning the nonlinear Fokker-Planck equation \eqref{eq:fokker planck intro limfluct}.
\begin{rem}
 \label{rem:invariance continuous model limfluct}\rm
It is immediate to see that \eqref{eq:fokker planck intro limfluct} exhibits the following symmetries: 
\begin{itemize}
 \item \emph{Rotation invariance:} if $q_t(\tta, \om)$ solves \eqref{eq:fokker planck intro limfluct} so does $q_t(\cdot+\tta_0, \om)$ for
any constant $\tta_0\in\Sun$,
\item \emph{Even symmetry:} if $q_t(-\tta, -\om)|_{t=0}=q_t(\tta, \om)|_{t=0}$, then it is true for all $t>0$.
\end{itemize}
In particular, the stationary solutions of \eqref{eq:fokker planck intro limfluct} will share these symmetries (see \eqref{eq:stationary solution introlimfluct}).
\end{rem}
\subsubsection{Synchronization and phase transition}
 As already observed by Sakaguchi (\cite{Sakaguchi1988}), the Kuramoto model exhibits a phase transition: if the coupling strength $K$ is small, the only stationary solution to \eqref{eq:fokker planck intro limfluct} is the \emph{incoherent solution} $q\equiv \frac{1}{2\pi}$, whereas for $K$ sufficiently large, the coupling dominates upon the thermal noise and non-constant stationary solutions exist. It is now well understood (see~\cite{Sakaguchi1988}) that crucial features of evolution \eqref{eq:fokker planck intro limfluct} are captured by the order parameters $r_{t}\geq0$ and $\psi_{t}\in\Sun$ (the continuous equivalents of $(r_{N}, \psi_{N})$ in \eqref{eq:intro order parameters N}) defined by: 
\begin{equation}
\label{eq:order parameter}
r_{t}e^{i\psi_{t}} = \intSR e^{i\tta}q_{t}(\tta, \om)\dd\tta\dd\mu(\om),\quad t\geq0.
\end{equation}
The quantity $r_{t}$ captures the degree of synchronization of a solution (the profile $q_t\equiv\frac{1}{2\pi}$ for example corresponds to $r_t=0$ and represents a total lack of synchronization) and $\psi_{t}$ identifies the center of synchronization: this is true and rather intuitive for unimodal profiles. In this framework, synchronization reads in the existence of nontrivial stationary solutions $q$ to \eqref{eq:fokker planck intro limfluct}: following \cite{Sakaguchi1988}, if $\mu$ is symmetric, any equilibrium in \eqref{eq:fokker planck intro limfluct} may be written as $q(\cdot+\tta_0, \om)$ for any fixed $\tta_0\in\Sun$ where
\begin{equation}
\label{eq:stationary solution introlimfluct}
 q(\tta, \om):= \frac{S(\tta, \om, 2Kr)}{Z(\om, 2Kr)},
\end{equation}
for
\begin{equation}
 S(\tta, \om, x):= e^{G(\tta, \om, x)}\left[ (1-e^{4\pi\om}) \int_{0}^{\tta}{e^{-G(u,\om, x)}\dd u} + e^{4\pi\om}\int_{0}^{2\pi}{e^{-G(u, \om,
x)} \dd u} \right],
\end{equation}
where $G(u, \om, x)= x\cos(u) + 2\om u$, $Z(\om, x) = \intS S(\tta, \om, x)\dd \tta$ a normalization constant. The parameter $r\in[0,1]$ in \eqref{eq:stationary solution introlimfluct} must satisfy the fixed-point relation \eqref{eq:order parameter}:
\begin{equation}
\label{eq:fixed point mu intro limfluct}
 r=\Psi_\mu(2Kr), \quad \text{where}\quad \Psi_\mu(x):= \intR\frac{\intS \cos(\tta) S(\tta, \om, x)\dd\tta}{Z(\om, x)}\mu(\dd\om).
\end{equation}
One can distinguish between two kinds of stationary solutions, depending on admissible solutions $r$ of \eqref{eq:fixed point mu intro limfluct}, :
\begin{itemize}
 \item $r=0$ is always a solution to \eqref{eq:fixed point mu intro limfluct} and corresponds to the constant density $q\equiv \frac{1}{2\pi}$,
\item Any solution $q$ with $r>0$ is called a synchronized solution. An easy calculation of the derivative of $\Psi_\mu(\cdot)$ at $0$ shows that such solutions exist at least when the coupling strength $K$ is greater than $\tilde{K}:=\left( \intR\frac{\mu(\dd\om)}{1+4\om^2} \right)^{-1}$. In that case, due to the rotation invariance (Remark~\ref{rem:invariance continuous model limfluct}), each solution $r>0$ of \eqref{eq:fixed point mu intro limfluct} corresponds to a whole circle of synchronized stationary solutions $\ens{q(\cdot+\tta_0, \om)}{\tta_0\in\Sun}$.
\end{itemize} 

\subsubsection{The case with no disorder}
\label{subsec:non disordered case intro limfluct}
In the non-disordered case ($\mu=\delta_0$), \eqref{eq:fokker planck intro limfluct} reduces to:
\begin{equation}
 \label{eq:fokker planck zero intro limfluct}
\partial_t q_t(\tta) = \half \partial_\tta^2 q_t(\tta) - \partial_{\tta} \left( q_t(\tta)(J \ast q_t)(\tta)\right),
\end{equation}
and any stationary profile can be written as $q_0(\tta+\tta_0)$ for
\begin{equation}
\label{eq:defq_0}
 q_0(\tta):=\frac{e^{2Kr_{0}\cos(\tta)}}{\int_{\Sun}e^{2Kr_0\cos(u)}\dd
u}=\frac{e^{2Kr_{0}\cos(\tta)}}{Z_0(2Kr_0)},
\end{equation}
where $r_0$ solves
\begin{equation}
\label{eq:Psizero limfluct}
r_0=\Psi_0(2Kr_0),\ \text{where}\ \Psi_0(x):=
\frac{\int_{\Sun}\cos(\tta) e^{x\cos(\tta)}\dd\tta}{\int_{\Sun}e^{x\cos(\tta)}\dd\tta}. 
\end{equation}
Here, since $\Psi_0$ is strictly concave (\cite[Lem.~4]{Pearce}) and $\partial_{r_0}\Psi_0(2Kr_0)=K$, the phase transition is obvious: for $K\leq 1$, $r_0=0$ is the only solution to \eqref{eq:Psizero limfluct} and $\frac{1}{2\pi}$ is the only stationary solution whereas for $K> 1$ this solution coexists with a unique (up to rotation) synchronized solution (corresponding to the unique $r_0>0$ solution to \eqref{eq:Psizero limfluct}).

\subsection{The evolution operator $L_{q}$}
The dynamics of the SPDE \eqref{eq:SPDE fluctuations limfluct} as $t\to+\infty$ is deeply linked to the spectral properties of the operator $L_q$ \eqref{eq:L linear qt intro limfluct}. We
will restrict ourselves to the stationary case, that is when $q|_{t=0}=q_t$ is equal to the synchronized (nontrivial) stationary solution $q$ \eqref{eq:stationary solution introlimfluct} of evolution \eqref{eq:fokker planck intro limfluct}. In this case, the object of interest is the stationary version of \eqref{eq:L linear qt intro limfluct}:
\begin{equation}
\label{eq:L linear q intro limfluct}
Lh(\tta, \om):= \half \partial_\tta^2 h(\tta, \om) - \partial_{\tta} \left(h(\tta, \om)\left(
\langle J \ast q\rangle_\mu(\tta) + \om\right) + q(\tta, \om) \langle J\ast h\rangle_\mu(\tta)\right)\, .
\end{equation}
The domain $\mD$ of the operator $L$ is given by:
\begin{equation}
 \label{eq:defdomainD}
\mD:=\ens{h(\tta, \om)}{\forall \om,
\tta\mapsto h(\tta, \om) \in \mC^{2}(\Sun),\ \intSR h(\tta, \om)\dd\tta\mu(\dd\om)=0}.
\end{equation}
\begin{rem}\rm
The choice of the domain $\mD$ of $L$ is crucial for the study of evolution \eqref{eq:SPDE fluctuations limfluct}. One encounters the same operator $L$ for the linear stability of the stationary solution $q$ since the linearized evolution of \eqref{eq:fokker planck intro limfluct} around $q$ is precisely given by $\partial_t h_t = Lh_t$. The natural domain for this latter evolution (see~\cite{GLP2011}) is 
\begin{equation}
\label{eq:domain zero means for all om intro limfluct}
\ens{h(\tta, \om)}{\forall \om,
\tta\mapsto h(\tta, \om) \in \mC^{2}(\Sun),\ \forall \om,\ \int_{\Sun}h(\tta, \om)\dd\tta=0}.
\end{equation}
Indeed for all $\om$, $q(\cdot, \om)$ is a probability density
on $\Sun$ so that perturbing by elements of domain \eqref{eq:domain zero means for all om intro limfluct} enables to remain within
the set of functions with integral $1$ on $\Sun$. Here, evolution \eqref{eq:SPDE fluctuations limfluct} does not live in domain \eqref{eq:domain zero means for all om intro limfluct} since $\eta$ has a nontrivial mean-value $C(\om)$ for fixed $\om$ (recall~\eqref{eq:struct etazero}). We will see that the non self-averaging phenomenon holds in \eqref{eq:defdomainD} and not in \eqref{eq:domain zero means for all om intro limfluct} (see Remark~\ref{rem:domain vs jordan limfluct}).
\end{rem}
\noindent
For the rest of this paper, we fix $K>1$ and we restrict ourselves to the case where
\begin{equation}
\label{eq:assumption mu delta pm omega}
\mu=\half\left( \delta_{-{\om_0}} + \delta_{\om_0} \right),
\end{equation}
where $\om_0>0$ is a fixed parameter. This assumption on $\mu$ appears to be quite restrictive, but generalizing parts of the results we present here to more general distributions $\mu$ does not seem to be straightforward. We refer to \S~\ref{subsec:comments} for a discussion on this topic. 

In what follows, the following standard notations will be used: for an operator $F$, we will denote by $\rho(F)$ the set of all complex numbers $\lambda$ for which $\lambda-F$ is invertible, and  by $R(\lambda, F):= \left( \lambda - F \right)^{-1}$, $\lambda\in\rho(F)$ the resolvent of $F$. The spectrum of $F$ will be denoted as $\sigma(F)$.
\medskip

The first goal of this paper is to state a spectral decomposition of the operator $L$ defined in \eqref{eq:L linear q intro limfluct},
based on perturbation arguments from the non-disordered case $\mu=\delta_0$ (see \S~\ref{subsec:non disordered case intro limfluct} and
\S~\ref{subsec:non disordered case limfluct}). 

\subsection{Distribution spaces}
\label{subsec:weighted sobo limfluct} The spectral analysis of the operator $L$ \eqref{eq:L linear q intro limfluct} will be mostly carried out in spaces of distribution that have
$H^{-1}$ regularity \wrt $\tta$. But the precise study of $L$ requires to introduce weighted version of $H^{-1}$ that we define here. We
first focus on weighted Sobolev spaces of functions $\tta\mapsto h(\tta)$ on $\Sun$ (\S~\ref{subsubsec:weighted sobo
limfluct}) and then introduce the corresponding spaces for functions with disorder $(\tta, \om)\mapsto h(\tta, \om)$ on $\Sun\times \Supp(\mu)$
(\S~\ref{subsubsec:space H}):
\subsubsection{Weighted Sobolev spaces}
\label{subsubsec:weighted sobo limfluct} 
For any bounded positive weight function $k(\cdot)$ on $\Sun$ such that $\intS k(\tta)\dd\tta=1$,
we may consider the space $\Ldk$ closure of $\mC(\Sun)$ \wrt the norm:
\begin{equation}
 \label{eq:norm Ldeux k limfluct}
\Ndk{h}:= \left( \intS \frac{h^2(\tta)}{k(\tta)} \dd\tta\right)^\frac{1}{2}.
\end{equation}
The decomposition of $h$ into the sum of $\Span(k)$ and its orthogonal supplementary in $\Ldk$ may be written as:
\begin{equation}
\label{eq:decomposition h hzero k}
h=\left( \intS h \right)\cdot k + h_0, 
\end{equation}
where $\intS h_0=0$. Since $h_0$ is with zero mean
value, each of its primitives are $2\pi$-periodic. In particular, we can consider $\bH^{-1}_{k}$ the closure of $\mC(\Sun)$ with respect to
the following weighted Sobolev norm:
\begin{equation}
 \label{eq:norm Hmoins 1 weight k limfluct}
\Vert h \Vert_{-1, k}:= \left( \left( \intS h \right)^2 + \intS\frac{\mH_0^2}{k}\right)^{\frac12},
\end{equation}
where $\mH_0$ is the primitive of $h_0$ on $\Sun$ such that $\intS\frac{\mH_0}{k}=0$. Note that one can understand the spaces $\bH^{-1}_{k}$ as part of a \emph{Gelfand-triple} construction (see~\ref{app:gelfand} for a precise definition). In particular, we will make a constant use of the space $\Hsq$ (that is for $k(\cdot)=q_0(\cdot)$ where $q_0$ is the stationary solution \eqref{eq:defq_0} of the non-disordered system) which is the natural space (see Prop.~\ref{prop:BGP}) for the study on the Kuramoto operator $L_{q_0}$ \eqref{eq:defL0} in the non-disordered case.
\begin{rem}
 \label{rem:weighted spaces no disorder limfluct}\rm
In the case of a constant weight $k(\cdot)\equiv \frac{1}{2\pi}$, we will write $(\Ld, \Nd{\cdot})$ and $(\bH^{-1}, \Vert\cdot\Vert_{-1})$
instead of $(\bL^2_{\frac{1}{2\pi}}, \Vert{\cdot}\Vert_{2, \frac{1}{2\pi}})$ and $(\bH^{-1}_{\frac{1}{2\pi}}, \Vert\cdot\Vert_{-1,
\frac{1}{2\pi}})$.
\end{rem}

\subsubsection{Weighted Sobolev spaces (with disorder)} \label{subsubsec:space H} The natural space in which to study the operator $L$ is
the space of functions $h$
in $\mD$ such that each component $h(\cdot, \om)$ lives in a certain $\bH^{-1}_{k(\cdot, \om)}$ for a weight $k(\cdot, \om)$ (which
may depend on $\om\in\Supp(\mu)$). More precisely, for any family of positive weight functions $\left(k(\cdot,
\om)\right)_{\om\in\Supp(\mu)}$, we denote as $\bH^{-1}_{\mu, k}$ the closure of $\mD$ \wrt the norm:
\begin{equation}
 \label{eq:norm Hmoins 1 weight k mu limfluct}
\Vert h \Vert_{\mu, -1, k}:= \left(\intR\Vert h(\cdot, \om) \Vert_{-1, k(\cdot, \om)}^2 \mu(\dd\om) \right)^\frac12=\left( \intR\left( \intS
h\dd\tta\right)^2\dd\mu + \intR\intS\frac{\mH_{0}^2}{k}\dd\tta\dd\mu\right)^{\frac12}.
\end{equation}
We will also consider the analogous averaged weighted $\bL^2$-spaces, that is the space $\Ldkmu$ given by the norm:
\begin{equation}
 \label{eq:Ltwo space averaged limfluct}
\Vert h\Vert_{\mu, 2, k}:= \left( \intR\intS \frac{h(\tta, \om)^2}{k(\tta, \om)}\dd\tta\dd\mu(\om)\right)^\frac12.
\end{equation}
\begin{rem}
 \label{rem:weighted spaces with disorder limfluct}\rm
In the particular case of $k(\cdot, \om)\equiv\frac{1}{2\pi}$ for all $\om\in\Supp(\mu)$, we will write \[\Hsmu:= \bH^{-1}_{\mu,
\frac{1}{2\pi}}\]
and the corresponding norm will be denoted as $\NH{\cdot}$. We will also write $(\Ldmu, \Ndmu{\cdot})$ instead of
$(\bL^2_{\mu, \frac{1}{2\pi}}, \Vert{\cdot}\Vert_{\mu, 2, \frac{1}{2\pi}})$.
\end{rem}
\noindent The main theorem concerning the operator $L$ (Theorem~\ref{theo:spectreLq}) will be stated in $\Hsmu$ for the ease of exposition but its proof will require the introduction of weighted Sobolev spaces $\bH^{-1}_{\mu, k}$ for nontrivial weights $k$. 
\subsection{The non-disordered case}
\label{subsec:non disordered case limfluct}
In the context of the Kuramoto model without disorder, the
linearized operator $L_{q_0}$ around stationary solution $q_0$ (see \S~\ref{subsec:non disordered case intro limfluct}), with domain
\begin{equation}
\label{eq:domain Dzero}
\mD_0:=\ens{u\in\mC^2(\Sun)}{\int_{\Sun}u=0} 
\end{equation}
is:
\begin{equation}
\label{eq:defL0}
 L_{q_0}u:= \half \partial_\tta^2u - \partial_\tta \left[  q_0 (J\ast u) + u(J\ast q_0)\right].
\end{equation}
\noindent In~\cite{BGP}, it is mainly proved that $L_{q_0}$ is essentially self-adjoint in $\Hsq$:
\begin{prop}[{\cite[Th.~1.8]{BGP}}]
\label{prop:BGP}
 $(L_{q_0}, \mD_0)$ is essentially self-adjoint in $\Hsq$. The spectrum of (the self-adjoint extension of) $L_{q_0}$
is pure point lying in $(-\infty, 0)$; $0$ is in the spectrum, with one-dimensional eigenspace (spanned by $\partial_\tta q_0$).
Moreover, the distance $\gapz$ between the eigenvalue $0$ and the rest of the spectrum is strictly positive.
\end{prop}
\subsection{Non self-averaging phenomenon for the operator $L$ and existence of a Jordan block}
Linear trajectories that depend on the initial condition as observed in Figure~\ref{subfig:evoleta limfluct} are reminiscent of an analogous
deterministic finite-dimensional example: consider the $2$-dimensional evolution $\svect{x'(t)}{y'(t)}= L \svect{x(t)}{y(t)}$, for
$L=\bigl(\begin{smallmatrix}0 & 1\\ 0 & 0\end{smallmatrix}\bigr)$. It is
trivial to see that the solutions of this system are linear in time: $\frac{x(t)}{t}\to y_0$ as $t\to\infty$. The existence of such a Jordan block is precisely equivalent to the existence of $x$ and $y$ such that $Lx=0$ and $Ly=x$. The purpose of the first main theorem of this paper is to prove an analogous existence of a Jordan block for the operator $L$ in \eqref{eq:L linear q intro limfluct}:

\begin{theo}
\label{th:existence jrodan block intro limfluct}
 For any fixed $\om_0>0$, if $q$ is the stationary solution in \eqref{eq:stationary solution introlimfluct}, then 
\begin{equation}
\label{eq:qprime kernel L}
 L\partial_\tta q=0.
\end{equation}
Moreover, there exists $ p\in\mD$ such that \begin{equation}
\label{eq:jordanhq'}
 \forall \tta\in\Sun, \forall  \om\in\Supp(\mu),\ \quad L p(\tta, \om) = \partial_\tta q(\tta, \om).
\end{equation}
In particular, the characteristic space of $L$ in $0$ is \emph{at least} of dimension $2$.
\end{theo}
\begin{rem}
\label{rem:domain vs jordan limfluct}\rm
Equality \eqref{eq:qprime kernel L} is a direct consequence of the rotation invariance in \eqref{eq:Kmod} (Remark~\ref{rem:invariance continuous model limfluct}). Note also that $p(\cdot, \om)$ found in \eqref{eq:jordanhq'} is with nontrivial mean value for all $\om\in\Supp(\mu)$. We believe in fact that $\intS p(\cdot, \om)=-\frac1\om$; this fact is derived from non-rigorous computations and verified by numerical simulations. In other terms, such a $p$ (and the corresponding Jordan block $\bigl(\begin{smallmatrix}0&1\\0&0\end{smallmatrix}\bigr)$ in the matrix representation \eqref{eq:matrix representation L H} of the operator $L$) do not exist on the domain \eqref{eq:domain zero means for all om intro limfluct}.
\end{rem}
\noindent Theorem~\ref{th:existence jrodan block intro limfluct} is proved in Section~\ref{sec:jordan block limfluct}.

\subsection{Spectral properties of $L$ and position of the spectrum}
The second goal of this paper is to prove that $L$ generates an analytic semi-group of operators with spectrum lying in the complex half-plane with negative real part:
\begin{theo}
\label{theo:spectreLq} 
In the Hilbert space $\Hsmu$ defined in Remark~\ref{rem:weighted spaces with disorder limfluct}, the operator $(L, \mD)$ is densely defined,
closable, its
closed extension having compact resolvent. In particular, its spectrum consists of isolated eigenvalues with finite multiplicities.

Moreover, for all $K>1$, for all $\alpha\in(0, \frac{\pi}{2})$, for all $\rho\in(0,1)$, there exists $\om_\star=\om_\star(K, \alpha, \rho)>0$
such that, for all $0<\om_0<\om_\star$, the following is true:

\begin{itemize}
\item The spectrum of $L$ lies in a cone $C_{\alpha}$ with vertex $0$ and angle $\alpha$
\begin{equation}
 C_{\alpha}\, :=\, \ens{\lambda\in\bbC}{\frac{\pi}{2} + \alpha\leq \arg(\lambda)\leq \frac{3\pi}{2}-\alpha}\subseteq \ens{z\in\bC}{\Re(z)\leq
0}\, ;
\end{equation}
\item There exists $\alpha'\in(0, \frac\pi2)$ such that $L$ is the infinitesimal generator of an analytic semi-group defined on a
sector $\Delta_{\alpha'}:=\{\lambda\in\bbC,\, |\arg(\lambda)|< \alpha'\}$;
\item the dimension of the characteristic space in $0$ is \emph{exactly} $2$, spanned by $\partial_\tta q$ and $p$, where $p$ is defined in
Theorem~\ref{th:existence jrodan block intro limfluct},
\item the eigenvalue $0$ is separated from the rest of the spectrum at a distance $\gapL=\lambda(L, K, \rho)$ at least equal to $\rho\cdot \min\left(\gapz, \frac12 e^{-4Kr_0}\right)$, where $L_{q_0}$ and $r_0$ are defined in \S~\ref{subsec:non disordered case intro limfluct}.
\end{itemize}
\end{theo}
\noindent Note that Theorem~\ref{theo:spectreLq} relies on perturbation arguments of the non-disordered case mentioned in \S~\ref{subsec:non disordered case limfluct}; in particular, the spectral gap $\gapL$ found in Theorem~\ref{theo:spectreLq} depends on the spectral gap $\gapz$ for the non-disordered case.

\medskip

As a consequence of Theorem~\ref{theo:spectreLq}, there exists a decomposition of $\Hsmu$ into the direct sum
\begin{equation}
 \label{eq:spaces projections H limfluct}
\Hsmu= G_0\oplus \Gneg,
\end{equation}
where $G_0$ is of dimension $2$ (spanned by $\partial_\tta q$ and $p$) such that the restriction of the operator $L$ to $G_0$ has spectrum
$\{0\}$ and the
restriction of $L$ to $\Gneg$ has spectrum $\sigma(L)\smallsetminus\{0\}\subseteq \ens{\lambda\in\bC}{\Re(\lambda)<0}$. We will denote as $P_0$
the corresponding projection on $G_0$ along to $\Gneg$, and $\Pneg=1-P_{0}$. In particular, there exist unique continuous linear forms
$\ell_{\partial_\tta q}$ and $\ell_p$ such that for all $h\in \Hsmu$
\begin{equation}
 \label{eq:ell projection P0}
P_0h := \ell_{\partial_\tta q}(h) \partial_\tta q + \ell_p(h) p.
\end{equation}
\noindent To fix ideas, one may think of the following infinite matrix representation for the operator $L$:
\begin{equation}
 \label{eq:matrix representation L H}
\begin{blockarray}{cccccccccccl}
\begin{block}{cc(c|c)c(cc|ccc)cl}
 &&\multirow{2}{*}{$\Pzero L \Pzero$}&\multirow{2}{*}{$\Pzero L \Pneg$}&& 0
&1&\BAmulticolumn{3}{c}{\ell_{\partial_\tta q}\left(L\Pneg\right)}&&\multirow{2}{*}{\hspace{-20pt}$\left.\begin{array}{c}
                                                                             \} \partial_\tta q\\\}p
                                                                             \end{array}
\right\}G_0$}\\
 &&&&&0 &0&0&\ \ \cdots&0 \\[3pt]\cline{3-4}\cline{6-10}
L&=&\multirow{3}{*}{$\Pneg L \Pzero$}&\multirow{3}{*}{$\Pneg L\Pneg$}&=&0&0&\BAmulticolumn{3}{c}{\multirow{3}{*}{$\quad\Pneg
L\Pneg\quad$}}&&\multirow{2}{*}{\hspace{-15pt}$\left.\begin{array}{c}
                                                                              \makebox[14pt][c]{}\\[18pt]\mbox{}
                                                                             \end{array}
\right\}\Gneg$}\\
&&&&&\vdots&\vdots&\\
&&&&&0 & 0  &   \\
  \end{block}
&&\underbrace{\vspace{-15pt}\makebox[40pt][c]{}}_{\substack{G_0=\\
\Span(\partial_\tta q,p)}}&\underbrace{\vspace{-15pt}\makebox[2cm][c]{}}_{\Gneg}
&&\underbrace { \vspace { -10pt } } _ { \partial_\tta q } &\underbrace{\vspace{-10pt}}_{p}
&\BAmulticolumn{3}{c}{\underbrace{\vspace{-10pt}\makebox[2cm][c]{}}_{\Gneg}}\\
\end{blockarray}
\end{equation}
Note that the second line in the matrix representation \eqref{eq:matrix representation L H} of $L$ is indeed equally zero since for all
$h\in \Hsmu$, $Lh$ is of zero mean value on $\Sun$; in particular $\ell_p(Lh)=0$ for all $h\in \Hsmu$.
\begin{rem}
 \label{rem:expression of lp}\rm
Any element $h=\left(h(\tta, \om)\right)_{\tta\in\Sun, \om\in\Supp(\mu)}$ can be identified in our binary case \eqref{eq:assumption
mu delta pm omega} with a couple $(\hp(\tta), \hm(\tta))_{\tta\in\Sun}$. Moreover, any $h\in \Hsmu$ can be decomposed according to
\eqref{eq:spaces projections H limfluct}:
\begin{equation*}
 h= \ell_{\partial_\tta q}(h) \partial_\tta q + \ell_p(h) p + \Pneg h.
\end{equation*}
Let us integrate the latter decomposition \wrt $\tta$. Since $\intS Lu=0$ for all $u\in\mD$, we have $\intS\Pneg h=0$ so that one can actually
find an explicit formulation for the functional $\ell_p$:
\begin{equation}
 \label{eq:expression of lp}
\ell_p(h)=\frac{\intS h_+}{\intS p_+}= \frac{\intS h_-}{\intS p_-}.
\end{equation}
The last equality in \eqref{eq:expression of lp} is due to the fact that $\intS \left(\hp+ \hm \right)=\intS \left(p_++ p_- \right)=0$.
\end{rem}
\subsection{Long time evolution of the fluctuation SPDE}

We now turn to the main result of the paper, which concerns the asymptotic behavior of the fluctuation process $\eta$ defined in
\eqref{eq:SPDE fluctuations limfluct}:

\begin{theo}
\label{theo:long time evolution eta limfluct}
Under the hypothesis of Theorem~\ref{theo:spectreLq}, there exists a unique weak solution $\eta$ to \eqref{eq:SPDE fluctuations limfluct} in
$\Hsmu$. Moreover, $\eta$ satisfies the following asymptotic linear behavior: for fixed initial condition $\eta_0^\om=X+C(\om)$, there exists $v(\om)\in\bR$ such that
\begin{equation}
 \label{eq:linear behavior eta}
\frac{\eta_t}{t}\substack{\text{in law}\\ \xrightarrow{\hspace*{30pt}}\\t\to\infty}=v(\om)
\partial_\tta q,\quad \text{as $t\to+\infty$}.
\end{equation}
Moreover, $\om\mapsto v(\om)$ is a Gaussian random variable with variance
\begin{equation}
\label{eq:variance speed intro}
 \sigma_v^2:= \left( 2\intS p_+(\tta)\dd\tta \right)^{-2},
\end{equation}
where $p_+(\tta):=p(\tta, \om_0)$ is defined by \eqref{eq:jordanhq'}.
\end{theo}

\subsection{Comments on Theorem~\ref{theo:long time evolution eta limfluct}}
\label{subsec:comments}
\subsubsection{Initial asymmetry of the disorder}
\label{subsec:asymdisor}
As we will see in the proof of Theorem~ \ref{theo:long time evolution eta limfluct}, the speed $v(\om)$ in \eqref{eq:linear behavior eta} depends explicitly on the mean-value of the initial condition $C(\om)$ (recall \eqref{eq:struct etazero}): $\eta_{0}^{\om}= X+ C(\om)$. Let us be more explicit on this dependence. At time $t=0$, for $N\geq1$ and $\vphi:\Sun\times\bR\to\bR$, $\eta_{N,0}(\vphi)$ defined by \eqref{eq:fluctuations process intro limfluct} may be written as
\begin{align}
 \eta_{N,0}(\vphi)&= \frac{1}{\sqrt{N}}\Som{j}{1}{N}{\left(\vphi(\tta_j, \om_j) - \intSR \vphi(\tta,
\om)\gamma(\dd\tta)\mu(\dd\om)\right)},\nonumber\\
&=\frac{1}{\sqrt{N}}\Som{j}{1}{N}{\left(\vphi(\tta_j, \om_j) - \intS \vphi(\tta,
\om_j)\gamma(\dd\tta)\right)}\nonumber\\&+\frac{1}{\sqrt{N}}\Som{j}{1}{N}{\left(\intS \vphi(\tta, \om_j)\gamma(\dd\tta)-\intSR \vphi(\tta,
\om)\gamma(\dd\tta)\mu(\dd\om)\right)},\label{eq:def CN}\\
&:= X_N(\vphi) + C_N(\vphi).\nonumber
\end{align}
The process $X_N$ captures the initial fluctuations of the rotators whereas $C_{N}$ captures the fluctuations of the disorder. It is easily seen that $C_{N}$ converges in law (\wrt the disorder) to the process $C$ with covariance given by \eqref{eq:covariance C}. As we will see in the proof of Theorem~\ref{theo:long time evolution eta limfluct}, $v(\cdot)$ actually depends on the process $C_+$ (indexed by functions $\psi:\Sun\to\bR$) that is the restriction of the process $C$ to the component on $+\om_0$ (recall \eqref{eq:assumption mu delta pm omega}):
\begin{equation}
 \label{eq:Cplus}
\forall \psi:\Sun\to\bR,\ C_{+,\psi}:=C_{\psi \ind{\om=\om_0}}.
\end{equation}
Thanks to \eqref{eq:def CN}, $C_+$ is the limit in law of the microscopic process $C_{N,+}$ defined by
\begin{equation}
\label{eq:intro asymmetry disorder}
 \forall \psi,\ C_{N,+}(\psi):= \left(\intS \psi(\cdot)\dd\gamma\right)
\frac{1}{\sqrt{N}}\Som{i}{1}{N}{\left(\ind{(\om_i=\om_0)}-\frac12\right)}:=\left(\intS \psi(\cdot)\dd\gamma\right)\frac{\alpha_N}{\sqrt{N}}.
\end{equation}
Here, $\alpha_N$ is exactly the (centered) number of frequencies among $(\om_1, \ldots, \om_N)$ that are positive, so that $C_{N,+}$ captures the lack of symmetry of the initial chosen disorder: $\alpha_N>0$ (resp. $\alpha_N<0$) represents the case of an asymmetry in favor of positive (resp. negative) frequencies. 

In~\cite[\S~10.2, p.~47]{Balmforth2000}, it is observed numerically that if we get rid artificially\footnote{for $N\geq1$, choose $2N$ frequencies by sampling the first $N$ according to $\mu$ and choose the $N$ remaining frequencies as the exact opposite of the first ones.} of the asymmetry between frequencies, there is no rotation in \eqref{eq:Kmod}, no matter how the frequencies are sampled.  We actually retrieve this phenomenon in Theorem~\ref{theo:long time evolution eta limfluct} in the case where $\mu=\frac12\left( \delta_{-\om_0}+\delta_{\om_0} \right)$, since in that case the quantity $\alpha_{2N}$ in \eqref{eq:intro asymmetry disorder} is equally zero for all $N\geq1$ and so is the consequent limit speed $v$.

\subsubsection{Perspectives}
One could hope to generalize the results of the paper in at least two directions. Firstly, we have restricted ourselves to the binary case $\mu=\frac12(\delta_{-\om_0}+ \delta_{\om_0})$. Note that the proof of Theorem~\ref{th:existence jrodan block intro limfluct} concerning the existence of a Jordan block (although written in this particular case for the reader's convenience) is not specific to this case: one could easily rewrite the same proof for more general distributions $\mu$ (even with unbounded support), satisfying appropriate integrability conditions in $0$ and in $\infty$. 

The main restriction on $\mu$ concerns Theorem~\ref{theo:spectreLq}: the hypothesis $\mu=\frac12(\delta_{-\om_0} +\delta_{\om_0})$ is critical
 for its proof. Indeed, the key argument of the proof is based on the fact that perturbing a
finite dimensional kernel of an operator $A$ by a sufficiently small perturbation $B$ leads to a kernel for the operator $A+B$ with the same finite dimension. But for distributions more general than \eqref{eq:assumption mu delta pm omega}, the kernel of $L$ is likely to become of infinite dimension, so that similar perturbation arguments cannot be applied. 

Secondly, Theorem~\ref{theo:spectreLq} is only proved for small disorder $\om_0$ whereas one would expect it to be true even for large
disorder. It is indeed natural to believe that the non self-averaging phenomenon seen in Figure~\ref{fig:nonself} not only holds for
large disorder but would even be more noticeable in that case. However, since Theorem~\ref{theo:spectreLq} relies on perturbation arguments, proving similar results for large $\om_0$ seems to require alternative methods.

\section{On the existence of a Jordan block for $L$ (Proof of Theorem~\ref{th:existence jrodan block intro limfluct})}

\label{sec:jordan block limfluct}
The purpose of this section is to prove Theorem~\ref{th:existence jrodan block intro limfluct}, i.e.  the fact that the operator $L$ defined in \eqref{eq:L linear q intro limfluct} has a Jordan block of size at least $2$. The symmetry of the system (Remark~\ref{rem:invariance continuous model limfluct}) leads to consider the set of distributions which are odd \wrt $(\tta, \om)\in\Sun\times\Supp(\mu)$:
\begin{equation}
\label{eq:mO odd functions}
\mO:=\ens{h}{\forall (\tta, \om)\in\Sun\times\Supp(\mu),\ h(-\tta, -\om)=-h(\tta,
\om)}.
\end{equation} 
We also denote by $\mN$ the set of functions with zero mean-value for all $\om\in\Supp(\mu)$ (recall the definition of $\mD$ in
\eqref{eq:defdomainD}):
\begin{equation}
 \label{eq:mN zero mean value limfluct}
\mN:=\ens{h\in\mD}{\forall \om\in\Supp(\mu), \intS h(\tta, \om)\dd\tta=0}.
\end{equation}
In the following straightforward lemma, whose proof is left to the reader, we sum-up the basic properties of the
stationary solution $q$ \eqref{eq:stationary solution
introlimfluct} and the operator $L$ \eqref{eq:L linear q intro limfluct}:
\begin{lem}
\label{lem:estimqLq}
The following statements are true:
 \begin{enumerate}
\item $\partial_\tta q\in\mO\cap\mN$,
\item \label{it:muzero} If $h\in \mO$ then $Lh \in \mO$,
\item For all $h\in\mD$, $Lh\in\mN$,
\item There exist $0<c<C$ such that for all $\tta\in\Sun$, $\om\in\Supp(\mu)$, $0<c\leq q(\tta, \om)\leq C$,
\item For all $\tta\in\Sun$, $\om\in\Supp(\mu)$,
\begin{equation}
\label{eq:integrated version of stationary equation limfluct}
 \half \partial_\tta q(\tta, \om) = q(\tta, \om)\left(  \langle J\ast q\rangle_\mu +\om\right) + \kappa(\om)\, ,
\end{equation}
where $\kappa(\om)=\frac{1 - e^{4\pi\om}}{2Z(\om)}$, and $Z(\om)=Z(\om, 2Kr)$ is the normalization constant defined in \eqref{eq:stationary
solution introlimfluct}.
 \end{enumerate}
\end{lem}
\noindent The fact that $\partial_\tta q\in\mO$ can be seen as a consequence of Remark~\ref{rem:invariance continuous model limfluct}. A direct calculation shows that $\partial_\tta q$ is in the kernel of $L$ (it corresponds to the rotation invariance of the problem). The rest of this section is devoted to prove the existence of an element $p\in\mD$ such that $Lp=\partial_\tta q$.

We recall here the definition of the weighted Sobolev spaces introduced in \S~\ref{subsubsec:space H}: we use here the spaces $\left(\bH^{-1}_{\mu, q}, \crosqqmu{\cdot}{\cdot}\right)$ defined in \eqref{eq:norm Hmoins 1 weight k mu limfluct} in the case of $k=q$ and $\left(\Ldmu, \crodmu{\cdot}{\cdot}\right)$ defined in Remark~\ref{rem:weighted spaces with disorder limfluct}. The main result is
the following:

\begin{prop}
\label{prop:jordLq}
For every $\om_0>0$, in the binary case \eqref{eq:assumption mu delta pm omega}, for every $v\in \bH^{-1}_{\mu, q}\cap \mO$ (and in particular
for
$v=\partial_\tta q$), there exists some $p\in \Ldmu\cap\mO$
such that
\begin{equation}
\label{eq:variational form for jordan limfluct}
 \forall l\in \bH^{-1}_{\mu, q}\cap \mO,\ \crosqqmu{L p}{l} = \crosqqmu{v}{l}.
\end{equation}
Moreover, in the case $v=\partial_\tta q$, any $p$ that satisfies \eqref{eq:variational form for jordan limfluct} is in fact a regular function
($p(\cdot, \om)\in\mC^\infty(\Sun)$ for all $\om\in\Supp(\mu)$) and is a classical solution to \eqref{eq:jordanhq'}.
\end{prop}

\begin{rem}\rm
 The scope of Proposition~\ref{prop:jordLq} is more general than the restrictive case of a binary law $\mu=\frac12\left( \delta_{-\om_0}+
\delta_{\om_0} \right)$; the following proof works for more general distributions $\mu$, the only additional requirement
being integrability conditions\footnote{those conditions are obviously satisfied in the binary case \eqref{eq:assumption mu delta pm omega}.}
in $0$ and $+\infty$, see Remark~\ref{rem:integrability conditions X limfluct}.
\end{rem}
\noindent Proof of Proposition~\ref{prop:jordLq} relies on several lemmas:
\begin{lem}
\label{lem:step one proof jord}
 For $h\in \mO\cap\mD$, $l\in \mO\cap\mD$, let us introduce the Dirichlet form 
\begin{equation}
\label{eq:Bhl}
 \mE_L(h, l):= \crosqqmu{Lh}{l}.
\end{equation}
$\mE_L(\cdot, \cdot)$ is well defined on $\mD(\mE_L):= \left(\Ldmu\cap\mO\right)\times \left(\bH^{-1}_{\mu, q}\cap \mO\right)$ and one can
decompose $\mE_L(\cdot, \cdot)$ into:
\begin{equation}
 \label{eq:Arielle}
\forall (h, l)\in\mD(\mE_L),\ \mE_L(h,l)= \Gamma(h,l) + K \ell(h)\cdot \ell(l),
\end{equation}
where $\Gamma(\cdot, \cdot)$, bilinear form on $\mD(\mE_L)$ and $\ell(\cdot)$ linear form on $\bH^{-1}_{\mu, q}\cap \mO$, are defined as
follows:
\begin{align}
\forall (h, l)\in \mD(\mE_L),\ \Gamma(h,l)&:=-\frac{1}{2} \intSR\frac{hl}{q}\dd\lambda\dd\mu +
\intSR \kappa(\cdot)\frac{h\mL}{q^2}\dd\lambda\dd\mu,\label{eq:defGammajord}\\
 \forall l\in \bH^{-1}_{\mu, q}\cap \mO,\ \ell(l)&:=\intSR l \sin(\cdot)\dd\lambda\dd\mu,\label{eq:ell jord limfluct}
\end{align}
where $\kappa$ and $\mL$ in \eqref{eq:defGammajord} are respectively defined in \eqref{eq:integrated version of stationary equation limfluct}
and as the primitive of $l\in\bH^{-1}_{\mu, q}$ such that $\intS \frac{\mL(\cdot, \om)}{q(\cdot, \om)}=0$ for all $\om\in\Supp(\mu)$ (recall
\S~\ref{subsubsec:space H}).
\end{lem}

\begin{lem}
 \label{lem:step two proof jord}
 For all continuous linear form $f$ on $\bH^{-1}_{\mu, q}\cap \mO$, there exists some $ p_1\in \Ldmu\cap\mO$ such that for $l\in
\bH^{-1}_{\mu, q}\cap \mO$
\begin{equation}
 \Gamma( p_1, l)= f(l). 
\end{equation}
\end{lem}
\begin{lem}
 \label{lem:step three proof jord}
The linear form $\ell(\cdot)$ defined in \eqref{eq:ell jord limfluct} can be expressed as a scalar product on
$\bH^{-1}_{\mu, q}\cap \mO$: there exists $p_2\in \mD\cap\mO$, for all $l\in \bH^{-1}_{\mu, q}\cap \mO$
\begin{equation}
\ell(l)=\crosqqmu{L p_2}{l}.
\end{equation}
\end{lem}
\noindent Let us admit for a moment Lemmas~\ref{lem:step one proof jord},~\ref{lem:step two proof jord} and~\ref{lem:step three proof jord} and
let us
prove Proposition~\ref{prop:jordLq}:

\begin{proof}[Proof of Proposition~\ref{prop:jordLq}]
Let $v$ be a fixed element of $\bH^{-1}_{\mu, q}\cap \mO$. Applying Lemma~\ref{lem:step two proof jord} to the continuous linear form
$f(l)=\crosqqmu{v}{l}$, there exists some $ p_1\in \Ldmu\cap\mO$ such that $\Gamma( p_1, l)= \crosqqmu{v}{l}$, which gives, using Lemma \ref{lem:step one proof jord} and Lemma~\ref{lem:step three proof jord}:
\begin{align*}
 \crosqqmu{v}{l}&= \Gamma(p_1, l),\\ &= \crosqqmu{L p_1}{l} - K \ell(p_1) \ell(l),\\&= \crosqqmu{L p_1}{l} - K \ell(p_1)\crosqqmu{L p_2}{l}.
\end{align*}
We can conclude that the variational formula \eqref{eq:variational form for jordan limfluct} is verified for the following choice of $p$:
\begin{equation}
\label{eq:defhh1h2}
 p:=  p_1 - K \ell(p_1)  p_2 \in \Ldmu\cap\mO.
\end{equation}
Let us prove now that such $p$ is in fact a regular function in $\tta$: since $p_2\in\mD$ is regular in $\tta$, it suffices to prove that
for all $\om\in\Supp(\mu)$, $\tta\mapsto p_1(\tta, \om)$ is $\mC^{2}$ (in fact $\mC^{\infty}$) in $\tta$. We start from the definition of
$p_1$:
\[\forall l\in \bH^{-1}_{\mu, q}\cap \mO,\quad \Gamma(p_1,l) = \crosqqmu{\partial_\tta q}{l}.\]
Since this true for all $l\in \bH^{-1}_{\mu, q}\cap \mO$, thanks to the expression of $\Gamma$ in \eqref{eq:defGammajord}, we obtain that for
any
fixed $\om\in\Supp(\mu)$, for Lebesgue-almost every $\tta\in\Sun$:
\begin{equation}
\label{eq:ptttahp}
\frac{1}{4} \frac{p_1(\tta, \om)}{q(\tta, \om)} = - \kappa(\om) \left(\int_{0}^{\tta}\frac{p_1(u, \om)}{q(u, \om)^{2}}\dd u \right) +
\int_{0}^{\tta} \frac{Q(u, \om)}{q(u, \om)}\dd u, 
\end{equation}
where $Q(\cdot, \om)$ is the primitive of $\partial_\tta q(\cdot, \om)$ such that $\intS \frac{Q(\cdot, \om)}{q(\cdot, \om)}=0$. Using that $q$ is bounded and $\mC^{\infty}$ in $\tta$ and that $p_1(\cdot, \om)\in \Ld$, we see that the primitive $\int_{0}^{\tta}\frac{p_1(u, \om)}{q(u, \om)^{2}}\dd u$ has a $\mC^{1}$ version. Thanks to \eqref{eq:ptttahp}, $p_1$ has a $\mC^{1}$ version. So, the right-hand side of \eqref{eq:ptttahp} is a least $\mC^2$, and so does $p_1$. The same repeated argument shows that $p_1$ is $\mC^{\infty}$ in $\tta$. That concludes the proof of Proposition~\ref{prop:jordLq}.
\end{proof}
\noindent It now remains to prove the three lemmas:
\begin{proof}[Proof of Lemma~\ref{lem:step one proof jord}]
Let us prove equality \eqref{eq:Arielle}: since $\mL$ is a primitive of $l$, one has\begin{align*}
\frac{1}{2} \int_{\Sun} \frac{(\partial_\tta h)\mL}{q} &= -\frac{1}{2} \int_{\Sun} \frac{h l}{q} + \frac{1}{2} \int_{\Sun}
\frac{h\mL}{q^2} \partial_\tta q\, .
        \end{align*}
Using \eqref{eq:integrated version of stationary equation limfluct}, for $\om\in\Supp(\mu)$\begin{align*}
    \frac{1}{2} \int_{\Sun}\frac{(\partial_\tta h)\mL}{q} &= -\frac{1}{2} \int_{\Sun} \frac{h l}{q} +
\int_{\Sun} \frac{h\mL}{q}\left( \langle J\ast  q\rangle_\mu (\cdot) +\om\right)+\kappa(\om) \int_{\Sun} \frac{h\mL}{q^2}.
   \end{align*}
Thanks to the expression of $Lh$ in \eqref{eq:L linear q intro limfluct}, we obtain
\begin{align}
 \mE_L(h, l) &= -\frac12 \intSR\frac{h l}{q} + \intSR \kappa(\cdot)\frac{h}{q^2}\mL - \intSR
\langle J\ast h\rangle_\mu\mL,\label{eq:mckexplint}\\&=\Gamma(h,l) - \intSR \langle J\ast h\rangle_\mu\mL.\nonumber
\end{align}
Lastly, integrating by parts the last term in \eqref{eq:mckexplint} and expanding the cosine function (recall $J(\cdot)=-K\sin(\cdot)$), we
obtain:
\begin{align*}
- \intSR \langle J\ast  h\rangle_\mu\mL &= K\left(\intSR \cos(\cdot) l \dd\lambda\dd\mu\right)\left(\intSR \cos(\cdot)
h\dd\lambda\dd\mu \right)\\
&+K \left(\intSR \sin(\cdot) l \dd\lambda\dd\mu \right)\left(\intSR \sin(\cdot) h \dd\lambda\dd\mu \right).
\end{align*}
But, since $l\in\mO$, the first term in the latter expression is zero. The result \eqref{eq:Arielle} follows.
\end{proof}

\begin{proof}[Proof of Lemma~\ref{lem:step two proof jord}]
In this proof, we use the following extension to Lax-Milgram Theorem:

\begin{prop}[{\cite[chap. III]{Showalter1997}}]
\label{theo:lionslaxmilgram}
 Let $\{\mH, |.|\}$ be a Hilbert space and $\{\mG, \N{.}\}$ a normed linear space. Suppose $\Gamma:\mH\times\mG\to \bR$ is bilinear and that
$\Gamma(\cdot, \vphi)$ is continuous for each
$\vphi\in\mG$. If there exists some constant $C>0$ such that
\begin{equation}
 \inf_{\N{\vphi}=1} \sup_{|h|\leq 1} |\Gamma(h, \vphi)|\geq C, \quad \text{(weak coercivity)},
\end{equation}
Then for each $f\in\mG'$ there exists some $ p\in\mH$ such that $\Gamma( p, \vphi)=f(\vphi)$ for all
$\vphi\in\mG$.
\end{prop}

The principle of the proof of Lemma~\ref{lem:step two proof jord} is to show that the bilinear function $\Gamma$ defined in
\eqref{eq:defGammajord} satisfies Proposition~\ref{theo:lionslaxmilgram} for 
\begin{align}
\mH:=\Ldmu\cap\mO,\quad \text{endowed with}\quad \N{\cdot}_{\mu, 2},\\ 
\mG:=\bH^{-1}_{\mu, q}\cap\mO\cap \bL^\infty(\Sun\times\bR),\quad \text{endowed with}\quad \N{\cdot}_{\mu, -1, q}.
\end{align}
Namely, we have the following:
\begin{enumerate}
 \item For each $l\in\mG$, $\Gamma( \cdot, l)$is continuous on $\Ldmu\cap\mO$: indeed, for the first term of $\Gamma(h, l)$, we have
\begin{align*}
\left|\intSR \frac{h l}{q} \dd\lambda\dd\mu\right|&\leq C\Ninf{l}\intSR |h|\dd\lambda\dd\mu\leq C \N{h}_{\mu, 2}.
\end{align*}
And for the second term, using the boundedness of $q$:
\begin{align*}
 \intSR \left|\kappa(\cdot)\frac{h}{q^2}\mL\right|\dd\lambda\dd\mu&\leq C\intSR \left|h\mL\right|\dd\lambda\dd\mu\leq C \N{h}_{\mu, 2}
\NsQmu{l}.
\end{align*}
\item $\Gamma$ is weakly coercive: let us fix $l\in \mG$ such that $\NsQmu{l}=1$.

Let us choose $h=g\mL\in\Ldmu\cap\mO$, where for all $ \om\in\Supp(\mu)$, $g(\cdot, \om)$ is
a $2\pi$-periodic function to be defined later. Then, by integration by parts in the equality \eqref{eq:defGammajord}
\begin{align}
 \Gamma(h, l) &= -\frac{1}{2}\intSR \frac{g}{q}l\mL + \intSR \kappa(\cdot)\frac{g}{q^{2}} \mL^2 = \intSR \left\{ \frac14 \partial_\tta
\left(\frac{g}{q}\right) + \kappa(\cdot)\frac{g}{q^2}\right\}\mL^2 \, .\label{eq:Gamma jordan final
form limfluct}
\end{align}
Consider now for fixed $\om\in\Supp(\mu)$ the following first order ODE, with periodic boundary condition:
\begin{equation}
 \label{eq:odejord}
\frac14\partial_{\tta} f(\cdot, \om) +\kappa(\om)\frac{f(\cdot, \om)}{q(\cdot, \om)} = \frac{1}{q(\cdot, \om)},\quad \text{with $f(0,
\om)=f(2\pi, \om)$} .
\end{equation}
Then for any $\om\in\Supp(\mu)\smallsetminus\{0\}$, an explicit calculation (left to the reader) shows that there exists a unique solution to
\eqref{eq:odejord}, $\tta\mapsto f(\tta, \om)$.
\begin{rem}\rm
In the case $\om=0$, \eqref{eq:odejord} reduces to $\frac14\partial_{\tta} f= \frac{1}{q_0}$ which is incompatible with
the condition $f(0)=f(2\pi)$, since $\intS \frac{1}{q_0}\dd\tta>0$: there is no such $2\pi$-periodic solution in the case $\om=0$.
\end{rem}
 Moreover, it is straightforward to see that $\N{\intR|f(\cdot, \om)|\dd\mu}_{\infty,\Sun}\leq C$, for some constant $C>0$.
\begin{rem}
 \label{rem:integrability conditions X limfluct}\rm
It is easy to see that $f(\cdot, \om)$ is not bounded as $\om\to0$ and $\om\to+\infty$; thus, for general distributions $\mu$, the same
control on $f$ requires additional integrability assumptions in $0$ and $+\infty$ (namely $\intR \max\left(\frac{1}{|\om|}, e^{c\om}\right)
\mu(\dd\om)<\infty$ for some constant $c>0$).
\end{rem}

If we choose $h$ such that $h = g\cdot\mL$ with $g(\cdot, \om) = q(\cdot, \om) f(\cdot, \om)$, we have
the following:
\begin{itemize}
 \item By construction of $f$, using \eqref{eq:odejord} in \eqref{eq:Gamma jordan final form limfluct}, $\Gamma(h, l)=\NsQmu{l}^2=1$,
\item $\N{h}^2_{\mu, 2} \leq C
\int_{\Sun} f^2 \frac{\mL^2}{q}\dd\lambda \leq C$. So,
$\sup_{\N{h}_{\mu, 2, 1}\leq 1} |\Gamma(h, l)|\geq \frac{1}{C}$, where $C$ is independent of $l\in \mG$ such that $\NsQmu{l}=1$.
\end{itemize}
Applying Proposition~\ref{theo:lionslaxmilgram}, we obtain the existence of some $p_1\in \Ldmu\cap\mO$
such that $\Gamma(p_1, l) = f(l)$, for all $l\in\mG$. But by density, this is also true for $l\in \bH^{-1}_{\mu, q}\cap\mO$.\qedhere
\end{enumerate}
\end{proof}

\begin{proof}[Proof of Lemma~\ref{lem:step three proof jord}]
Since there exists some constants $C, c>0$ such that for all $ \om\in\Supp(\mu)$, $\tta\in \Sun$, $0< c\leq q(\tta, \om)\leq C$, $\ell$ is
continuous on $\bH^{-1}_{\mu, q}\cap\mO$ (as well as on $\Ldmu\cap\mO$). More precisely, by Riesz theorem, there exists a unique $e\in
\bH^{-1}_{\mu, q}\cap\mO$ such that for all $l\in \bH^{-1}_{\mu, q}\cap\mO$,
$\ell(l)=\crosqqmu{e}{l}$. One can be more explicit: a simple calculation shows that this $(\tta, \om)\mapsto
e(\tta, \om)$ corresponds to the primitive $\mE(\tta, \om)= - q(\tta, \om) \cos(\tta)$, that is:
\begin{equation}
\label{eq:jordwoe}
\forall \tta\in\Sun,\  \om\in\Supp(\mu), \quad e(\tta, \om) = - \partial_\tta q(\tta, \om) \cos(\tta) + q(\tta, \om)\sin(\tta). 
\end{equation}
Let us introduce the following function $p_2\in \Ldmu\cap\mO$:
\begin{equation}
\label{eq:defhh2}
p_2(\tta, \om) = \frac{e^{-B(\tta, \om)}}{1-e^{4\pi\om}}  \int_{\Sun} e^{B(u, \om) +4\pi\om} \dd u + \int_{0}^{\tta} e^{B(u, \om) -
B(\tta, \om)}\dd u,
\end{equation}
for\[B(\tta, \om) = -2 \left( Kr \left(\cos(\tta)-1\right)+\om\tta\right).\] Then one readily verifies that $Lp_2$
is proportional to $e$.
\end{proof}

\section{Global spectral properties of operator $L$ (Proof of Theorem~\ref{theo:spectreLq})}
\label{sec:global localization spectrum L limfluct}
The purpose of this section is to prove Theorem~\ref{theo:spectreLq}. The main idea of the proof is to decompose the operator $L$ defined by
\eqref{eq:L linear q intro limfluct} on the domain $\mD$ given by
\eqref{eq:defdomainD} into the sum of a self-adjoint operator $A$ (in a weighted Sobolev space for appropriate weights) and a
perturbation $B$ which will be considered to be small \wrt $A$. Namely, one can decompose $L$ \eqref{eq:L linear q intro limfluct} into $L=A+B$ where, for all $h\in\mD$, for all $ \om\in\Supp(\mu)$,
\begin{equation}
 \label{eq:operator A L limfluct}
Ah(\tta, \om):= \frac12 \partial_{\tta}^2h(\tta, \om) -\partial_\tta\Big(h(\tta, \om) (J\ast q_0)(\tta) + q_0(\tta) \langle J\ast
h\rangle_\mu \Big),\end{equation}
and,
\begin{equation}
 \label{eq:operator B L limfluct}
Bh(\tta, \om):= -\partial_\tta\Big(h(\tta, \om) \{\langle J\ast (q -q_0)\rangle_\mu(\tta) + \om\} + (q(\tta, \om)-q_0(\tta)) \langle J\ast
h\rangle_\mu \Big).
\end{equation}
We divide the proof of Theorem~\ref{theo:spectreLq} into three parts: in \S~\ref{subsec:spectral properties A limfluct}, we prove that $A$ is
essentially self-adjoint (and thus generates an analytic semigroup) in some weighted Sobolev space (recall \S~\ref{subsec:weighted sobo
limfluct}) for an appropriate choice of weights. Note that this section strongly relies on the fact that $\mu$ is a binary distribution.

The purpose of \S~\ref{subsec:control perturb B} is to establish precise control of the size of the perturbation $B$ \wrt $A$. The last step of the proof (\S~\ref{subsec:spectral properties of operator L limfluct d}) consists in deriving similar spectral properties for $L=A+B$, especially the fact that the spectrum of $L$ lies in the complex half-plane with negative real part.

\subsection{Spectral properties of the operator $A$}
\label{subsec:spectral properties A limfluct}
In this paragraph, we prove mainly that $A$ defined in \eqref{eq:operator A L limfluct} is essentially self-adjoint for a Sobolev norm that
is equivalent to the norm $\NH{\cdot}$ defined in \S~\ref{subsubsec:space H}.

Since we are working in the domain $\mD$ (recall \eqref{eq:defdomainD}), the test functions $h$ are such that $\intR h(\cdot, \om)\dd\mu
=\frac12\left(h(\cdot, +\om_0) +h(\cdot, -\om_0)\right)$ has zero mean value on $\Sun$. The idea of this paragraph is to reformulate the
operator
$A$ in terms of the sum $\frac12 (h(\cdot, +\om_0)+h(\cdot, -\om_0))$ and the difference $\frac12(h(\cdot, +\om_0)-h(\cdot, -\om_0))$; namely, we define the following $2\times2$ invertible matrix:
\[M:=\half \begin{pmatrix} 1 & 1\\ 1&-1\end{pmatrix},\]
and for $h\in\mD$, let $\svect{u}{v}:= M\cdot h$, namely
\begin{equation}
\label{eq:defuv}
\left\lbrace \begin{array}{ccl}
u(\cdot)&:=& \half(h(\cdot, +\om_0)+ h(\cdot, -\om_0)),\\
v(\cdot)&:=& \half(h(\cdot, +\om_0)- h(\cdot, -\om_0)).
\end{array}\right. 
\end{equation}
We are now able to define the following operator: $\tA:= M\circ A \circ M^{-1}$, defined on the domain $\tmD$
\begin{equation}
 \label{eq:defdomainDtilde}
\tmD:=\ens{(u, v)\in\mC^{2}(\Sun)\times\mC^{2}(\Sun)}{\int_{\Sun} u(\tta)\dd\tta=0},
\end{equation}
given by
\begin{align}
\forall (u, v)\in \tmD,\quad \tA\vect{u}{v}&:=\vect{\tA_1u}{\tA_2v}:=\left(\begin{aligned}
\half \partial_{\tta^2}u &- \partial_\tta  \left[ u (J\ast q_0) + q_0 (J\ast u)\right]\\
\half \partial_{\tta^2}v &- \partial_\tta\left[ v (J\ast q_0) \right]
\end{aligned}\right).\label{eq:defAselfadj}
\end{align}
The remarkable observation is that operator $\tA$ is now \emph{uncoupled} w.r.t. variables $u$ and $v$; consequently, in order to diagonalize $\tA$, it suffices to diagonalize both components of $\tA$, namely $\tA_1$ and $\tA_2$. This is the purpose of Propositions~\ref{prop:A1 semigroup} and \ref{prop:A2selfadj} below.

\medskip

We use here the weighted Sobolev norms $\Vert\cdot\Vert_{-1, k}$ defined in \eqref{eq:norm Hmoins 1 weight k limfluct} for different choices of $k(\cdot)$. Concerning the first component, $\tA_1=L_{q_{0}}$ (with domain $\{u\in\mC^2(\Sun),\ \intS u=0\}$) is equal to the McKean-Vlasov operator with no disorder defined in \eqref{eq:defL0}. Following \S~\ref{subsec:non disordered case limfluct}, the natural space for the study of $\tA_1$ is $\Hsq$ defined in \eqref{eq:norm Hmoins 1 weight k limfluct}, for the weight $k(\cdot)=q_0(\cdot)$ (recall \eqref{eq:defq_0}). In this space, we have
\begin{prop}
\label{prop:A1 semigroup}
In $\Hsq$, $\tA_1$ is essentially self-adjoint with compact resolvent and spectrum in the negative part of the real axis. $0$ is a
one-dimensional eigenvalue, spanned by $\partial_\tta q_0$. The spectral gap $\gapAf=\gapz$ between $0$ and the rest of the spectrum is
strictly
positive. 

Moreover, the self-adjoint extension of $\tA_1$ is the infinitesimal generator of a strongly continuous semi-group of contractions $\tT_1(t)$
on
$\Hsq$. For every $0<\alpha<\frac{\pi}{2}$, this semigroup can be extended to an analytic semigroup $\tT_1(z)$ defined
on $\Delta_\alpha=\ens{\lambda}{|\arg(\lambda)|<\alpha}$ and one has the following estimate on its resolvent (where $\Sigma_\alpha\, =\,
\ens{\lambda \in\bC}{|\arg(\lambda)|<\frac{\pi}{2}+\alpha} \cup \{0\}$):
\begin{equation}
\label{eq:estim resolvent A1 limfluct}
\forall \alpha\in\Big(0, \frac\pi2\Big),\ \forall \lambda\in\Sigma_\alpha,\ \Nsq{R(\lambda, \tA_{1})}\, \leq\, \frac{1}{1-\sin(\alpha)}\cdot
\frac{1}{|\lambda|}\, .
\end{equation}
\end{prop}
The second component $\tA_2$ is a second order ordinary differential operator, with domain $\mC^2(\Sun)$. The natural space in which to study $\tA_2$ (see \S~\ref{subsubsec:diagA2}) is $\HsPhi$, for the choice of the weight function
$\tta\mapsto w(\tta)=\frac{e^{-\Phi(\tta)}}{\intS e^{-\Phi}}$, with 
\begin{equation}
\label{eq:weight Phi limfluct}
\Phi(\tta):= -2Kr_0\cos(\tta),
\end{equation}
where $r_0$ is given by \eqref{eq:Psizero limfluct}. Namely, we have
\begin{prop}
\label{prop:A2selfadj}
 The operator $(\tA_2, \mC^2(\Sun))$ is essentially self-adjoint in $\HsPhi$ and has compact resolvent. Hence, its spectrum
consists of isolated eigenvalues with finite multiplicities. The kernel of $\tA_2$ is of dimension $1$, spanned by
$w(\tta)=\frac{e^{-\Phi(\tta)}}{\intS e^{-\Phi}}$. Moreover, we have the following spectral
gap estimation:
\begin{equation}
 \label{eq:estimgapA2}
\forall v\in \mC^2(\Sun), \quad -\crosPhi{\tA_2v}{v} \geq \frac{e^{-4Kr_0}}{2} \NsPhi{v - \left( \intS v \right) w},
\end{equation}
so that the spectrum of $\tA_2$ lies in the negative part of the real axis and the distance between $0$ and the rest of the spectrum
$\gapAs$ is at least equal to $\frac{e^{-4Kr_0}}{2}$. One also has explicit estimate on the resolvent of $\tA_2$:
\begin{equation}
\label{eq:estim resolvent A2 limfluct}
\forall \alpha\in(0, \frac\pi2),\ \forall \lambda\in\Sigma_\alpha,\ \NsPhi{R(\lambda, \tA_{2})}\, \leq\, \frac{1}{1-\sin(\alpha)}\cdot
\frac{1}{|\lambda|}\, .
\end{equation}
\end{prop}
Putting things together, the natural norm for the operator $\tA=(\tA_1, \tA_2)$ is the Hilbert-norm: $\left(\Nsq{u}^2+\NsPhi{v}^{2}\right)^{\frac{1}{2}}$, $(u,v)\in \tmD$. But since $\tA$ is the conjugate of $A$ through the invertible matrix $M$, to say that $\tA$ is essentially
self-adjoint
for the previous norm is equivalent to say that $A$ is essentially self-adjoint for the corresponding conjugate norm:
\begin{equation}
\label{eq:norm for A limfluct}
\forall h\in\mD,\ \NHw{h}:= \left(\Nsq{\frac12\left(h(\cdot, +\om_0) +h(\cdot, -\om_0)\right)}^2 +\NsPhi{\frac12\left(h(\cdot,
+\om_0)-h(\cdot, -\om_0)\right)}^{2} \right)^{\frac{1}{2}}.
\end{equation}
The results of \S~\ref{subsec:spectral properties A limfluct} can be summed-up in the following proposition, which is an easy consequence of
Propositions~\ref{prop:A1 semigroup} and~\ref{prop:A2selfadj}:

\begin{prop}
 \label{prop:Aselfadj}
For the norm $\NHw{\cdot}$ defined in \eqref{eq:norm for A limfluct}, the operator $(A, \mD)$ is essentially self-adjoint, with compact
resolvent. The spectrum of (the self-adjoint extension of) $A$ is pure-point, and consists of eigenvalues with finite
multiplicities. Moreover it lies in the negative part of the real-axis and $A$ is the infinitesimal generator of an analytic semigroup of
operators $T_A(z)$ defined on a domain
$\Delta_\alpha=\ens{z\in\bC}{|\arg(z)|<\alpha}$, for any $0<\alpha<\frac{\pi}{2}$. One also has the following estimate about the resolvent of
$A$:
\begin{equation}
\label{eq:estim resolvent A limfluct}
\forall \alpha\in\left(0, \frac\pi2\right),\ \forall \lambda\in\Sigma_\alpha,\ \NHw{R(\lambda, A)}\, \leq\, \frac{1}{1-\sin(\alpha)}\cdot
\frac{1}{|\lambda|}\, .
\end{equation}
The kernel of $A$ is of dimension $2$, spanned by $\left\{ \partial_\tta q_{0}+\frac{e^{-\Phi}}{\intS e^{-\Phi}},
\partial_\tta q_0-\frac{e^{-\Phi}}{\intS e^{-\Phi}}
\right\}$ and the eigenvalue $0$ is separated from the rest
of the spectrum with a distance $\gapA:=\min\left( \gapAf, \gapAs\right)$, where $\gapAf$ and $\gapAs$ are defined in
Propositions~\ref{prop:A1 semigroup} and~\ref{prop:A2selfadj} respectively.
\end{prop}

\begin{rem}\rm
The norm $\NHw{\cdot}$ is equivalent to the norm $\NH{\cdot}$ defined in \S~\ref{subsubsec:space H}, since the weights
$q_0$ and $w$ are bounded above and below. In $\Hsmu$, the operator $A$ (although no longer self-adjoint) still
generates an analytic semi-group with the same spectrum and the same spectral gap.
\end{rem}
\noindent The aim of paragraphs \S~\ref{subsubsec:diagA1} (resp. \S~\ref{subsubsec:diagA2}) is to prove Proposition~\ref{prop:A1 semigroup}
(resp.
Proposition~\ref{prop:A2selfadj}).
\subsubsection{Spectral properties of $\tA_1$: proof of Proposition~\ref{prop:A1 semigroup}}
\label{subsubsec:diagA1}
As $\tA_1=L_{q_{0}}$ corresponds to the linear evolution operator of the non-disordered Kuramoto model studied in~\cite{BGP}, we know from Proposition~\ref{prop:BGP} that $\tA_1$ is essentially self-adjoint and
dissipative in $\Hsq$. It remains to prove that $\tA_1$ generates an analytic semigroup $\tT_1(t)$ in an appropriate domain. We refer to
classical references \cite{Hille1957, Lunardi1995, Pazy1983} for detailed definitions of analytic semigroups of operators defined on a sector of the complex plane. We recall the following result about analytic extensions of strongly continuous semigroups. 
\begin{prop}[{\cite[Th 5.2, p.61]{Pazy1983}}]
 \label{prop:pazysemgps}
Let $T(t)$ a uniformly bounded strongly continuous semigroup, whose infinitesimal generator $F$ is such that $0\in\rho(F)$ and let
$\alpha\in(0, \frac\pi2)$. The
following statements are equivalent:
\begin{enumerate}
 \item \label{it:prop:pazysemgps1}$T(t)$ can be extended to an analytic semigroup
in the sector $\Delta_\alpha\, =\, \ens{\lambda\in\bC}{|\arg(\lambda)|<\alpha}$ and $\Vert{T(z)}\Vert$ is uniformly bounded in every closed
sub-sector $\bar{\Delta}_{\alpha'}$, $\alpha'<\alpha$, of $\Delta_\alpha$,
\item \label{it:prop:pazysemgps2} There exists $M>0$ such that\begin{equation}\rho(F) \supset \Sigma_\alpha\, =\,
\ens{\lambda
\in\bC}{|\arg(\lambda)|<\frac{\pi}{2}+\alpha} \cup \{0\},\end{equation}and\begin{equation}\Vert{R(\lambda, F)}\Vert \, \leq\, 
\frac{M}{|\lambda|}, \quad
\lambda\in\Sigma, \lambda\neq0\, .\end{equation}
\end{enumerate}
\end{prop}

We are now in position to prove the rest of Proposition~\ref{prop:A1 semigroup}: we know from~\cite[Prop. 2.3, Prop. 2.6]{BGP} that for any
$\lambda>0$, $\lambda-\tA_1$ is positive with range $\Hsq$. 
Consequently, we can apply Lumer-Phillips Theorem (see~\cite[Th 4.3 p.14]{Pazy1983}): $\tA_1$
is the infinitesimal generator of a $C_0$ semi-group of contractions denoted by $\tT_1(t)$.

The rest of the proof is devoted to show the existence of an analytic extension of this semigroup in a proper sector. We follow
here the lines of the proof of~\cite[Th 5.2, p. 61-62]{Pazy1983}, but with explicit estimates on the resolvent: let us first replace the
operator $\tA_1$ by a small perturbation: for all $\eps>0$, let $\tA_{1,\eps}:= \tA_1-\eps$, so that
$0$ belongs to $\rho(\tA_{1,\eps})$. As $\tA_1$, the operator $\tA_{1, \eps}$ is self-adjoint and generates a strongly continuous semigroup of operators (which is $\tT_{1,\eps}(t)=\tT_1(t)e^{-\eps t}$). Moreover, since $\tA_{1, \eps}$ is self-adjoint, we have
\begin{equation}
\label{eq:estimRA1eps1}
\forall
\lambda\in\bC\smallsetminus\bR, \Nsq{R(\lambda, \tA_{1,\eps})}\leq
\frac{1}{|\Im(\lambda)|}, 
\end{equation}
and since the spectrum of $\tA_{1, \eps}$ is negative, for every $\lambda\in\bC$ such that $\Re(\lambda)>0$, we have 
\begin{equation}
\label{eq:estimRA1epsdef1}
\Nsq{R(\lambda, \tA_{1,\eps})}\leq \frac{1}{|\lambda|}.                                     
\end{equation}
Let us prove that for $\lambda\in\Sigma_\alpha$,
\begin{equation}
\label{eq:estim resolvent A1 eps limfluct}
 \Nsq{R(\lambda, \tA_{1, \eps})}\, \leq\, \frac{1}{1-\sin(\alpha)}\cdot \frac{1}{|\lambda|}\, .
\end{equation}
Note that \eqref{eq:estim resolvent A1 eps limfluct} is clear from \eqref{eq:estimRA1eps1} and \eqref{eq:estimRA1epsdef1} when $\Re(\lambda)\geq0$. Let us prove it for $\lambda\in\Sigma_\alpha$ with $\Re(\lambda)<0$. Consider $\sigma>0, \tau\in\bbR$ to be chosen
appropriately later and write the following Taylor expansion for
$R(\lambda, \tA_\eps)$ around $\sigma+i\tau$ (at least well defined in a neighborhood of $\sigma+i\tau$ since $\sigma>0$):
\begin{equation}
\label{eq:taylor resolvent}
 R(\lambda, \tA_{1, \eps}) \, =\,  \sum_{n=0}^{\infty}{R(\sigma+i\tau, \tA_{1, \eps})^{n+1}((\sigma + i\tau)-\lambda)^n}\, .
\end{equation}
This series $R(\cdot, \tA_{1, \eps})$ is well defined in $\lambda\in\Sigma_\alpha$ with $\Re(\lambda)<0$ if one can choose $\sigma$, $\tau$ and
$k\in(0,1)$ such that $\Nsq{R(\sigma
+i\tau, \tA_{1, \eps})}|\lambda-(\sigma+i\tau)|\leq k<1$. In particular, using \eqref{eq:estimRA1eps1}, it suffices to have
$|\lambda-(\sigma+i\tau)|\leq k|\tau|$ and since $\sigma>0$ is
arbitrary, it suffices to find $k\in(0,1)$ and $\tau$ with $|\lambda-i\tau|\leq k|\tau|$ to obtain the convergence of \eqref{eq:taylor
resolvent}.
For this $\lambda\in\Sigma_\alpha$ with $\Re(\lambda)<0$, let us define $\lambda'$ and $\tau$ as in Figure~\ref{fig:angle alpha}. Then,
$|\lambda-i\tau|\leq |\lambda'-i\tau|= \sin(\alpha)|\tau|$ with $\sin(\alpha)\in(0,1)$. So
the series converges for $\lambda\in\Sigma_\alpha$ and one has, using again \eqref{eq:estimRA1eps1},
\begin{equation}
 \label{eq:estimRAepsdef}
\Nsq{R(\lambda, \tA_{1, \eps})} \, \leq\,  \frac{1}{(1-\sin(\alpha))|\tau|} \, \leq\,  \frac{1}{1-\sin(\alpha)}
\cdot\frac{1}{|\lambda|}\, .
\end{equation}
\begin{figure}[!ht]
 \centering
\includegraphics[width=0.5\textwidth]{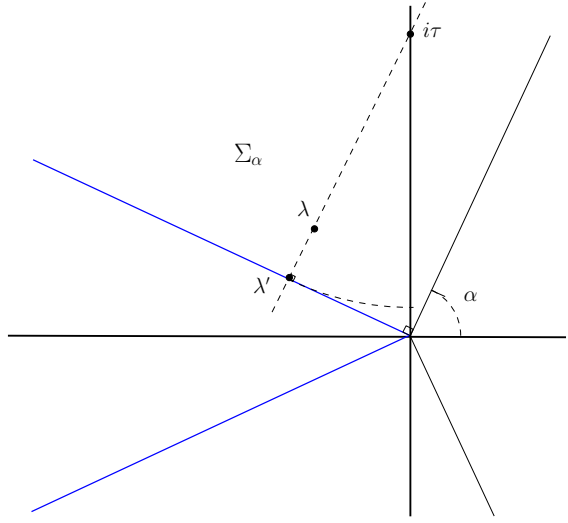}
\caption{The set $\Sigma_\alpha$.}
\label{fig:angle alpha}
\end{figure}
The fact that $\tT_{1, \eps}(t)$ can be extended to an analytic semigroup $\tilde{T}_{1, \eps}(z)$ on the domain $\Delta_\alpha$ is a
simple application of \eqref{eq:estim resolvent A1 eps limfluct} and Proposition~\ref{prop:pazysemgps}, with $M:=\frac{1}{1-\sin(\alpha)}$.
Let us then define $\tilde{T}_{1}(z):= e^{\eps z} \tilde{T}_{1, \eps}(z)$, for $z\in\Delta_\alpha$ so that $\tilde{T}_{1}(z)$ is
an analytic extension of $\tT_{1}(t)$ (an argument of analyticity shows that $\tilde{T}_{1}(z)$ does not depend on $\eps$).
\medskip

\noindent 
Note that estimation \eqref{eq:estim resolvent A1 limfluct} can be obtained by letting $\eps\to0$ in \eqref{eq:estim resolvent A1 eps
limfluct}.\qed

\subsubsection{Spectral properties of $\tA_2$: proof of Proposition~\ref{prop:A2selfadj}}
\label{subsubsec:diagA2}

$\tA_2$ may be written as \begin{equation}
\label{eq:defA2}
 \tA_2v= \half \partial_\tta^2v + \partial_\tta\left( vKr_0\sin(\cdot)\right),            
\end{equation}
where $r_0=\Psi_0(2Kr_0)$ (recall \eqref{eq:Psizero limfluct}). One recognizes in $\tA_2$ a  Fokker-Planck operator on $\mC^2(\Sun)$ with a sine potential. This operator can easily be seen, by integrations by parts in an appropriate weighted $\bL^2$-space, as a Sturm-Liouville operator (\cite{Dunford1988,Coddington1955}) acting on $\mC^2$, $2\pi$-periodic functions. The problem is that a $\bL^2$-norm is not appropriate for the future study of the SPDE \eqref{eq:SPDE fluctuations limfluct}: a look at the covariance structure of the noise $W$
(see \eqref{eq:W process intro limfluct}) shows that $W$ naturally lives in a $\bH^{-1}$-space instead of a $\bL^2$-space.

\medskip

An easy calculation shows that $\tA_2$ can be rewritten in terms of the weight function $\Phi$ defined in \eqref{eq:weight Phi limfluct}:
\begin{equation}
\label{eq:A2 partial Phi}
\tA_2v = \half\partial_\tta\left( e^{-\Phi} \partial_\tta\left( e^{\Phi}v \right) \right). 
\end{equation}
Let $w$ be:
\begin{equation}
 \label{eq:v0}
w(\tta):= \frac{e^{-\Phi(\tta)}}{\intS e^{-\Phi}}.
\end{equation}
One directly sees from \eqref{eq:A2 partial Phi} that $w$ lies in the kernel of $\tA_2$: $\tA_2w=0$. We are now in position to prove Proposition~\ref{prop:A2selfadj}: we place ourselves in the framework of the weighted Sobolev spaces $\left(\LdPhi, \crodPhi{\cdot}{\cdot}\right)$ and $\left(\HsPhi, \crosPhi{\cdot}{\cdot}\right)$, in the particular case of $k(\cdot)=w(\cdot)$. 
\begin{proof}[Proof of Proposition~\ref{prop:A2selfadj}]
In $\HsPhi$, the operator $(\tA_2, \mC^2(\Sun))$ is formally symmetric: for $u$, $v\in \mC^2(\Sun)$, for $u_0$ ad $v_0$ defined by
\eqref{eq:decomposition h hzero k}, we have successively,
\begin{align}
 \crosPhi{\tA_2u}{v} &= \left(\intS e^{-\Phi}\right)\intS e^{\Phi}\left( \frac 12 e^{-\Phi} \partial_\tta\left(
e^{\Phi}u \right)\right) \mV_0= -\frac 12 \left(\intS e^{-\Phi}\right)\intS  e^{\Phi}u v_0,\nonumber\\
&= -\frac 12 \intS  \frac{u_0 v_0}{w} -\frac 12 \intS u \intS v_0= -\frac 12 \intS 
\frac{u_0 v_0}{w}.\label{eq:A2symmetric}
\end{align}
\noindent Let us prove that $(\tA_2, \mC^2(\Sun))$ is essentially self-adjoint: let $\mE_2$ be the following Dirichlet form
\begin{equation}
\label{eq:dirichlet form for A2}
 \mE_2(u,v):= \crosPhi{u}{(1-\tA_2)v}= \intS u \intS v + \intS \frac{\mV_0 \mU_0}{w} + \frac12 \intS \frac{u_0 v_0}{w}.
\end{equation}
Then it is easy to see that $\mE_2$ is a continuous bilinear form on $\LdPhi$ (thanks to Poincar\'e inequality). Moreover $\mE_2$ is coercive: for all $u\in\LdPhi$
\begin{align}
 \mE_2(u,u) &= \left( \intS u \right)^2 + \intS \frac{\mU_0^2}{w}+ \frac12 \intS \frac{u_0^2}{w},\nonumber\\
&\geq \left( \intS u \right)^2 + \frac12 \NdPhi{u-\left( \intS u \right)w}^2\geq \frac12 \NdPhi{u}^2.\label{eq:estim bellow Dir form A2
limfluct}
\end{align}
Since for all $f\in \HsPhi$, the linear form $v\mapsto \crosPhi{v}{f}$ is continuous on $\LdPhi$, an application of
Lax-Milgram Theorem shows that for such an $f\in \HsPhi$ there exists an unique $u\in\LdPhi$ such that for all $v\in\LdPhi$
\begin{equation}
 \mE_2(v, u) = \crosPhi{v}{f}.
\end{equation}
It is then easy to see that $\intS f= \intS u$ and that for almost every $\tta\in\Sun$,
\begin{equation}
\label{eq:relation u f mF}
\frac12 \frac{u_0(\tta)}{w(\tta)} = -\int_{0}^{\tta}  \frac{\mF_0}{w} + \int_{0}^{\tta}\frac{\mU_0}{w}.
\end{equation}
Since $u\in\LdPhi$, $\mU_0$ admits a $\mC^1$-version and if we assume that $f$ is square-integrable, the same argument holds for the first
term of the right-hand side of \eqref{eq:relation u f mF}. So, if $f$ is square integrable, $u_0$ admits a $\mC^2$-version.
To sum-up, if we suppose that $f$ is continuous, there exists $u\in\mC^2(\Sun)$ such that, applying $\partial_\tta\left( e^{-\Phi}
\partial_\tta(\cdot) \right)$ to \eqref{eq:relation u f mF}:
\begin{align}
 f&= f_0 + \left( \intS f \right)w = -\frac12 \partial_\tta\left( e^{-\Phi} \partial_\tta\left( e^{\Phi} u_0\right) \right) +u_0 +
p_w(u)w,\nonumber\\
&= -\tA_2u_0 + u = (1-\tA_2)u.
\end{align}
But since those functions $f$ are dense in $\HsPhi$, we see that the range of $1-\tA_2$ is dense so that $\tA_2$ is essentially
self-adjoint.

Secondly, the spectral gap estimation \eqref{eq:estimgapA2} holds: for every $u\in\mC^2(\Sun)$, we have using
\eqref{eq:A2symmetric} and Poincar\'e inequality:
\begin{align*}
-\crosPhi{\tA_2v}{v} &= \frac 12 \left( \intS e^{-\Phi} \right)\intS e^{\Phi} v_0^2,\\
&\geq \frac 12 e^{-2Kr_0}\left( \intS e^{-\Phi} \right) \intS  \mV_0^2\geq \frac 12 e^{-4Kr_0} \left( \intS e^{-\Phi} \right) \intS 
e^{\Phi} \mV_0^2,\\
&= \frac 12 e^{-4Kr_0} \NsPhi{v - \left( \intS v \right)w}^2.
\end{align*}
Moreover, $\tA_2$ has compact resolvent: it suffices to prove that $\lambda-\tA_2$ has compact resolvent for at least one value of
$\lambda$. We prove it for $\lambda=1$ which is indeed in the resolvent set, thanks to the beginning of this proof.
For $u\in \HsPhi$, let us consider $f:= (1-\tA_2)^{-1}u$ so that $\crosPhi{f}{(1-\tA_2)f} = \crosPhi{f}{u}$. Using the coercivity of $\mE_2$, one has, $c\NdPhi{f}^{2}\leq \crosPhi{f}{u} \leq \NsPhi{f}\NsPhi{u}$, for some constant $c$. Using the continuous injection of $\LdPhi$ into $\HsPhi$ (say $\NsPhi{\cdot} \leq C \NdPhi{\cdot}$, for some positive constant $C$), one has
\begin{equation}
 \NdPhi{f}\leq \frac{C}{c}\NsPhi{u}.
\end{equation}
So $(1-\tA_2)^{-1}$ maps sequences that are bounded in $\HsPhi$ into sequences that are bounded in $\LdPhi$. It remains then to prove
that the injection of $\LdPhi$ into $\HsPhi$ is compact. This is indeed true since for every $v\in \HsPhi$, one has, by
Cauchy-Schwartz inequality \[|\mV_0(\tta) - \mV_0(\tta')| \leq C\NdPhi{v_0} \sqrt{|\tta -\tta'|}\leq C\NdPhi{v} \sqrt{|\tta
-\tta'|}.\] That means that, by Ascoli-Arzela Theorem that the sets $\ens{v\in \HsPhi}{\NdPhi{v}\leq cst}$ are relatively compact
in $\mC(\Sun)$ and also in $\LdPhi$. That completes the proof.

The fact that $\tA_2$ generates an analytic semigroup $\tT_2(z)$ on the same sector $\Delta_\alpha$ as well as estimation \eqref{eq:estim
resolvent A2 limfluct} can be derived in the same way as in \S~\ref{subsubsec:diagA1}. That concludes the proof of Proposition
\ref{prop:A2selfadj}.
\end{proof}

\subsection{Control on the perturbation $B$}
\label{subsec:control perturb B}
In order to derive spectral properties for the operator $L=A+B$, we need to have a precise estimation about the smallness of the perturbation
$B$ \wrt operator $A$ studied in the previous paragraph \S~\ref{subsec:spectral properties A limfluct}. 

\begin{rem}\rm
For simplicity, we work now with the norm $\NH{\cdot}$ (recall Remark~\ref{rem:weighted spaces with disorder limfluct}); as already
mentioned this norm is equivalent to the norm $\NHw{\cdot}$ used in \S~\ref{subsec:spectral properties A limfluct}. Recall also the definition
of the space $(\Ld, \Nd{\cdot})$ defined in Remark~\ref{rem:weighted spaces no disorder limfluct} and of $(\Ldmu, \Ndmu{\cdot})$
defined in Remark~\ref{rem:weighted spaces with disorder limfluct}.

Secondly, since the whole operator $L$ is no longer symmetric in $\Hsmu$, its spectrum need not be real. Thus, we will assume for the
rest of this document that we work with the complexified versions of the scalar products defined previously in this paper. The results
concerning the operator $A$ are obviously still valid.
\end{rem}
The smallness of the perturbation $B$ with respect to $A$ can be quantified in terms of the difference $\Ninf{q(\cdot, \om)-q_0(\cdot)}$, $
\om\in\Supp(\mu)$.
For the ease of exposition, we do not attempt to derive precise estimations of this difference $\Ninf{q(\cdot, \om)-q_0(\cdot)}$ (Lemma~\ref{lem:estimalphaR}) and of coefficients $a(\om_0)$ and $b(\om_0)$ (Lemma~\ref{prop:Bperturb}), in terms of the coupling strength $K$. $c$ will be a positive constant (depending on $K$) which may change from a line to another.
\begin{lem}
\label{lem:estimalphaR}
 For $\om>0$ and $K>1$, let us define
\begin{equation}
 \label{eq:Norme inf epsilon}
 \Ninf{q-q_0}:= \sup_{\tta\in\Sun,\, |u|\leq \om} |q(\tta, u)-q_0(\tta)|.
\end{equation}
Then $\Ninf{q-q_0}= O(\om)$, as $\om\to0$.
\end{lem}
\begin{proof}
This is clear since one can bound $\partial_\om q(\tta, \om)$ uniformly in $(\tta, \om)$, as $\om\to0$ (by a constant depending on $K$).
\end{proof}
\begin{prop}
 \label{prop:Bperturb}
The operator $B$ is $A$-bounded in the sense that there exist positive constants $a(\om_0)=a(\om_0, K)$ and $b(\om_0)=b(\om_0, K)$ such that
\begin{equation}
 \label{eq:BAbounded}
\forall h\in\mD,\quad \NH{Bh}\leq a(\om_0) \NH{h} + b(\om_0)\NH{Ah},
\end{equation}
and moreover, for fixed $K>1$,
\begin{equation}
 \label{eq:aombom}
a(\om_0)=O(\om_0),\ \text{and}\ b(\om_0)=O(\om_0),\quad \text{as}\ \om_0\to 0, 
\end{equation}
\end{prop}
\begin{proof}[Proof of Proposition~\ref{prop:Bperturb}]
Recall that $\langle h \rangle_\mu(\cdot)= \intR h(\cdot, \om)\mu(\dd\om)$ is the averaging of $h(\cdot, \om)$. The proof consists in two steps: we first prove that there exists some
constant $\alpha_{K, \om_0}$ such that for all $h\in \mD$,
\begin{align}
 \NH{Bh} &\leq \alpha_{K, \om_0}\Ndmu{h},\label{eq:estimB1}
\end{align}
Indeed, for given $h\in \mD$, for all $ \om\in\Supp(\mu)$, we have $\Vert Bh(\cdot, \om)\Vert_{-1}= \Nd{\mB h(\cdot, \om)}$,
where $\mB h(\cdot, \om)$ is the appropriate primitive of $Bh(\cdot, \om)$ in $\bH^{-1}$, (recall Remark~\ref{rem:weighted spaces no disorder limfluct}):
\begin{align*}
\mB h&:={} -h \left( \langle J\ast (q -q_0)\rangle_\mu + \om\right) - (q-q_0)\cdot \langle J\ast h\rangle_\mu\\ 
&+{}\intS h \cdot\left( \langle J\ast (q -q_0)\rangle_\mu +\om\right) +\intS (q-q_0)\cdot \langle J\ast h\rangle_\mu.
\end{align*}
Using the boundedness of $q_0$ and the bounds $|\left( J\ast \eps \right)|\leq 4 K \Ninf{\eps}$ and $|\langle J\ast h\rangle_\mu|\leq
\frac{K}{\sqrt{2}}\Nd{\langle h\rangle_\mu}$, it is easy to deduce that, for some constant $c>0$:
\begin{equation}
\label{eq:estimmB1}
|\mB h| \leq c(\Ninf{q-q_0} +\om_0) (|h|+ \Nd{\langle h \rangle_\mu}).
\end{equation}
Consequently,
\begin{align}
 \NH{Bh} &= \left(\langle\Nd{\mB h}^2\rangle_\mu\right)^\frac12\leq c(\Ninf{q-q_0} +\om_0) \Ndmu{h},
\end{align}
so that \eqref{eq:estimB1} is satisfied for some coefficient $\alpha_{K, \om_0}$ such that (Lemma~\ref{lem:estimalphaR}) $\alpha_{K,
\om_0}=O_{\om_{0}\to0}(\om_0)$.

\medskip

The second step of the proof is to control the $\bL^2$-norm $\Ndmu{h}$ of $h$ with the $\Hsmu$-norms of
$A h$ and $h$: namely we prove that there exist constants $\gamma_K$ and $\delta_K$ such that
\begin{align}
 \Ndmu{h} &\leq \gamma_{K}\NH{Ah} + \delta_{K}\NH{h}.\label{eq:estim u A1u d to s}
\end{align}
The proof is based on a usual interpolation argument: for all integer $n>1$, for any $f\in\mC^2(\Sun)$, one has
\begin{align}
\label{eq:identuu'u''}
 \Nd{\partial_\tta f}^2\leq \sqrt{n}\Nd{f} \frac{\Nd{\partial_\tta^2f}}{\sqrt{n}}\leq \frac n2 \Nd{f}^2 + \frac{\Nd{\partial_\tta^2f}^2}{2n}.
\end{align}
Let us use this interpolation relation \eqref{eq:identuu'u''} to derive \eqref{eq:estim u A1u d to s}: for all $h\in\mD$, $ \om\in\Supp(\mu)$,
one has
\begin{equation}
 \Nd{h(\cdot, \om)}^2 = \left( \intS h(\cdot, \om) \right)^2 + \Nd{h_{0}(\cdot, \om)}^2
\end{equation}
Applying relation
\eqref{eq:identuu'u''} with $f(\cdot)=\mH_{0}(\cdot, \om)$ we obtain
\begin{align}
 \Nd{h(\cdot, \om)}^2&\leq \left|\intS h(\cdot, \om)\right|^2 + \frac n2 \Nd{\mH_{0}(\cdot, \om)}^2 + \frac{\Nd{\partial_\tta h(\cdot,
\om)}^2}{2n},
\end{align}
where we used the fact that $\partial_\tta h_{0}(\cdot, \om)=\partial_\tta h(\cdot, \om)$. Integrating \wrt $\mu$,
\begin{equation}
 \label{eq:estvv'}
\Ndmu{h}^2 \leq \intR\left|\intS h\right|^2\dd\mu + \frac n2\Ndmu{\mH_{0}}^2 + \frac{\Ndmu{\partial_\tta h}^2}{2n}.
\end{equation}
As previously for the operator $B$, a simple calculation shows that for all $ \om\in\Supp(\mu)$, we have
$\Vert Ah(\cdot, \om)\Vert_{-1}= \Nd{\mA h(\cdot, \om)}$, where $\mA h$ is the appropriate primitive of $Ah$ in $\bH^{-1}$
(recall \eqref{eq:operator A L limfluct}):
\begin{equation}
\mA h = {} \half \partial_\tta h - h(J\ast q_0) - q_0 \langle J\ast h\rangle_\mu +\intS \left( h (J\ast q_0) + q_0 \langle J\ast
h\rangle_\mu\right),
\label{eq:defmA2}
\end{equation}
so that, for some constant $c>0$
\begin{equation}
\label{eq:estim hprime wrt Ah two limfluct}
\Ndmu{\partial_\tta h}^2 \leq 12 \Ndmu{\mA h}^2 + c\Ndmu{h}^2.
\end{equation}
Injecting this inequality in \eqref{eq:estvv'}, one obtains
\begin{equation}
\label{eq:hsig vs Ahsig bbE limfluct}
\Ndmu{h}^2 \leq \intR\left|\intS h\right|^2\dd\mu + \frac n2 \Ndmu{\mH_{0}}^2 + \frac{1}{2n}\left(12\Ndmu{\mA h}^2 +
c\Ndmu{h}^2 \right).
\end{equation}
Choosing $n>1$ sufficiently large so that the coefficient in front of $\Ndmu{h}^2$ in the right-hand side of \eqref{eq:hsig vs Ahsig
bbE limfluct} is lower than $\frac12$ leads to (for some constant $c>0$):
\begin{align*}
\Ndmu{h}^2 &\leq 2\intR\left|\intS h\right|^2\dd\mu + c\Ndmu{\mH_{0}}^2 + c\Ndmu{\mA h}^2\leq c\NH{h}^2 + c \NH{Ah}^2,
\end{align*}
which shows \eqref{eq:estim u A1u d to s}. Putting together estimates \eqref{eq:estimB1} and \eqref{eq:estim u A1u d to s}, we find the
$A$-boundedness of $B$ \eqref{eq:BAbounded} with coefficients $a(\om_0)$ and $b(\om_0)$ which satisfy \eqref{eq:aombom}, thanks to Lemma
\ref{lem:estimalphaR}.\qedhere
\end{proof}

\begin{prop}
\label{prop:B Acompact}
The operator $B$ is $A$-compact, in the sense that for any sequence $(h_p)_{p\geq 0}\in\mD^{\bN}$ such that $\NH{h_p}$ and $\NH{Ah_p}$
are bounded, there exists a convergent subsequence for $(Bh_p)_{p\geq1}$.
\end{prop}
\begin{proof}[Proof of Proposition~\ref{prop:B Acompact}]
Let $(h_p)_{p\geq 0}$ a sequence in $\mD$ such that $\NH{h_p}$ and $\NH{Ah_p}$ are bounded by a constant $c$. A closer look at the operator $B$
defined in \eqref{eq:operator B L limfluct} and the definition of the norm $\NH{\cdot}$ in \eqref{eq:norm Hmoins 1 weight k mu
limfluct}
shows that it suffices to prove that there exists a subsequence $(h_{p_k})_{k\geq0}$ that is convergent in $\Ldmu$. In
particular, for all
$p\geq 0$, $\NH{Ah_p}\leq c$. Using this boundedness and \eqref{eq:estim hprime wrt Ah two limfluct}, we have
$\Ndmu{\partial_\tta h_{p}}\leq c + c\Ndmu{h_p}$, so that
\begin{align}
\Ndmu{\partial_\tta h_{p}}&\leq c +c\left(\intR\left|\intS h\right|^2\dd\mu + \Ndmu{h_{0,p}}^2\right),\\
&\leq c +c\left(\intR\left|\intS h\right|^2\dd\mu + \frac n2\Ndmu{\mH_{0, p}}^2 + \frac{\Ndmu{\partial_\tta h_{p}}^2}{2n}\right),
\end{align}
where we used again \eqref{eq:identuu'u''} for $f=\mH_{0, p}(\cdot, \om)$. Choosing a sufficiently large $n>1$ leads to
$\Ndmu{\partial_\tta h_{0, p}}=\Nd{\partial_\tta h_{p}}\leq c+c\NH{h_p}\leq c$ for a constant $c$ independent of $p\geq0$. An easy application
of Cauchy-Schwartz
inequality leads to $| h_{0, p}(\tta, \om)- h_{0, p}(\tta', \om)| \leq \Ndmu{\partial_\tta h_{0, p}} \sqrt{|\tta-\tta'|}$, for all
$\om\in\{\pm\om_0\}$.
Since
the functions $(h_{0, p})_{p\geq 0}$ are such that $\intS h_{0, p}=0$ for all $p\geq 0$, Ascoli-Arzela Theorem and the previous bound show
the existence of a convergent subsequence $(h_{0, p_k})$ (for each $ \om\in\Supp(\mu)$) in the space of continuous functions on $\Sun$. In
particular, this subsequence is convergent in $\Ldmu$ and is renamed $(h_{0, p})_{p\geq0}$, with a slight abuse of notations.

The fact that $\intR|\intS h_{p}|\dd\mu\leq c$ shows that one can extract a further subsequence of $(h_p)_{p\geq0}$ which is also convergent in
$\Ldmu$.
This concludes the proof.
\end{proof}

\subsection{Spectral properties of $L=A+B$}
\label{subsec:spectral properties of operator L limfluct d}
We are now in position to derive by perturbation results on $A$ similar spectral properties on $L=A+B$ using theory of perturbation of operators (\cite{Kato1995}) and analytic semi-groups (\cite{Pazy1983}).
\subsubsection{The spectrum of $L$ is pure point}
\begin{prop}
 \label{prop:Ltildecompactresolvent}
For all $K$, for all $\om_0>0$,
\begin{enumerate}
  \item  the operator $(L, \mD)$ is closable. In that case, its closure has the same domain as the closure of $A$,
 \item the closure of $L$ has compact resolvent. In particular, its spectrum is pure point.
\end{enumerate}
\end{prop}
\begin{rem}\rm
 Note that Proposition~\ref{prop:Ltildecompactresolvent} is valid without any assumption on the smallness of $\om_0$, since it relies on the
relative compactness of $B$ with respect to $A$ (Prop.~\ref{prop:B Acompact}).
\end{rem}

\begin{proof}[Proof of Proposition~\ref{prop:Ltildecompactresolvent}]
It is a simple consequence of the relative compactness of $B$ \wrt the self-adjoint operator $A$. The first assertion of Proposition
\ref{prop:Ltildecompactresolvent} is a consequence of~\cite[Th. 1.11 p.194]{Kato1995} and the second assertion can be found in~\cite[Lemma
3.6, p.17]{Markus1988} for example.
\end{proof}

\subsubsection{$L$ generates an analytic operator}

We prove that the perturbed operator $L$ generates an analytic semigroup of operators on a
appropriate sector. An immediate corollary is the position of the spectrum in a cone whose vertex is zero. We know (Proposition~\ref{prop:Aselfadj}) that for all $0<\alpha<\frac{\pi}{2}$, $A$ generates a semigroup of operators on $\Delta_\alpha=\ens{z\in\bC}{|\arg(z)|<\alpha}$.
\begin{prop}
 \label{prop:semigroupAB}
For all $K>1$, $\eps>0$ and $0<\alpha<\frac{\pi}{2}$, there exists $\om_{1}>0$ (depending on $\alpha$, $K$ and $\eps$) such that
for all $0<\om_0<\om_{1}$, the spectrum of $L$ lies within the sector
$\Theta_{\eps,\alpha}:=\ens{\lamb\in\bC}{\frac{\pi}{2} +\alpha \leq \arg(\lamb)\leq \frac{3\pi}{2}-\alpha}
\cup\ens{\lamb\in\bC}{ |\lamb|\leq \eps}.$
For such $\om_0$, $L$ generates an analytic semigroup of operators on $\Delta_{\alpha'}$, for some $\alpha'\in(0,
\frac\pi2)$.
\end{prop}
\begin{proof}[Proof of Proposition~\ref{prop:semigroupAB}]
Let $0<\alpha<\frac{\pi}{2}$ be fixed. Thanks to \eqref{eq:estim resolvent A limfluct}, there exists a constant $c>0$ (which comes from the
equivalence of the norms $\NH{\cdot}$ and $\NHw{\cdot}$) such that for every
$\lamb\in\Sigma_\alpha:=\ens{\lamb\in\bC}{|\arg(\lamb)|<\frac{\pi}{2}+\alpha}$:\[\NH{R(\lamb, A)}
\leq \frac{c}{(1-\sin(\alpha))|\lamb|}\quad \text{and,} \quad \NH{AR(\lamb, A)} \leq 1+ \frac{c}{(1-\sin(\alpha))}.\]
Then for $\lamb\in\Sigma_\alpha$, $h\in\mD$:
\begin{align*}
 \NH{BR(\lamb, A)h}&\leq a(\om_0) \NH{R(\lamb, A)h} + b(\om_0)\NH{AR(\lamb, A)h},\\
&\leq \left( a(\om_0) \frac{c}{(1-\sin(\alpha))|\lamb|} +
b(\om_0)\left( 1+ \frac{c}{1-\sin(\alpha)} \right) \right)\NH{h}.
\end{align*}
Let us fix $\eps>0$ and choose $\om_1$ such that:
\begin{equation}
 \max\left(\frac{4a(\om_1)c}{\eps(1-\sin(\alpha))}, 4 b(\om_1) \left( 1+ \frac{c}{1-\sin(\alpha)}\right)\right) \leq 1.
\end{equation}
For this choice of $\om_1$, for all $0<\om_0<\om_1$, for any $\lamb\in\Sigma_\alpha$ such that $|\lamb|\geq \eps\geq
\frac{4a(\om_1)c}{1-\sin(\alpha)}$, we have $\NH{BR(\lamb, A)h}\leq \frac{1}{2}\NH{h}$, and thus the operator $1 - BR(\lamb, A)$ is
invertible with $\NH{1- BR(\lamb, A)}\leq 2$. Since it can easily be shown that \[R(\lamb, A +B) = R(\lamb, A) \left( 1 -
BR(\lamb, A) \right)^{-1},\] one deduces the following estimates about the resolvent of the perturbed operator $L= A +B$:
\begin{equation}
\label{eq:estim resolvent L}
\forall \lamb\in\Sigma_\alpha,\ |\lamb|\geq \eps,\ \NH{R(\lamb, L)}\leq \frac{2}{(1-\sin(\alpha))|\lamb|}. 
\end{equation}
The fact that the spectrum of $L$ lies within $\Theta_{\eps, \alpha}$ is a straightforward consequence of \eqref{eq:estim resolvent L}.
Secondly, \eqref{eq:estim resolvent L} entails that $L$ generates an analytic semigroup of operators on an appropriate
sector. Indeed, if one denotes by $L_{\eps}:= L -\eps$, one deduces from \eqref{eq:estim resolvent L} that
$0\in\rho(L_{2\eps})$ and that for all $\lambda\in\bC$ with $\Re(\lambda)>0$ (in particular, $|\lambda|<|\lambda+2\eps|$)
\begin{align}
 \NH{R(\lambda, L_{{2\eps}})}\, &=\, \NH{R(\lambda+2\eps, L)}\, \leq\, \frac{2}{(1-\sin(\alpha))|\lambda+2\eps|}\,
,\nonumber\\
&\, \leq\, \frac{2}{(1-\sin(\alpha))|\lambda|}\, .
\end{align}
Hence, using the same arguments of Taylor expansion as in the proof of Proposition~\ref{prop:A1 semigroup} and applying Proposition
\ref{prop:pazysemgps}, one easily sees that $L_{{2\eps}}$ generates an analytic semigroup in a (a priori) smaller sector
$\Delta_{\alpha'}$, where $\alpha'\in(0, \frac\pi2)$ can be chosen as $\alpha':= \frac12 \arctan\left( \frac{1-\sin(\alpha)}{2} \right)$. But
if
$L_{2\eps}$ generates an analytic semigroup, so does $L$.
\end{proof}

\subsubsection{$0$ is an eigenvalue with multiplicity $2$}
\label{subsubsec:loczero}

Let us fix $K>1$, $\alpha\in(0, \frac{\pi}{2})$, $\rho\in(0,1)$ and define $\eps=\rho \gapA$. Applying Proposition~\ref{prop:semigroupAB}, we
know that for small $\om_0$ (depending on $K$, $\alpha$, $\rho$), $L$ generates an analytic semigroup on $\Theta_{\eps, \alpha}$.
Let $\Theta_{\eps, \alpha}^{+}:= \ens{\lambda\in\Theta_{\eps, \alpha}}{\Re(\lambda)\geq 0}$ be the subset of
$\Theta_{\eps, \alpha}$ which lies in the positive part of the complex plane (see Figure~\ref{fig:position spectrum A L limfluct}). 

The purpose of this paragraph is to show that one can choose a perturbation $B$ small enough so that no non-zero eigenvalue of $A+B$
remains in the small set $\Theta_{\eps, \alpha}^{+}$. 

\medskip 

To do so, we proceed by an argument of local perturbation: we know that the distance $\gapA= \min\left( \gapAf, \gapAs\right)$ between the
eigenvalue $0$ and the rest of the spectrum of $A$ is strictly positive. In particular, one can separate $0$ from the rest of the spectrum of
$A$
by a circle $\mathscr{C}$ centered in $0$ with radius $\left( \frac{\rho+1}{2} \right)\gapA$.
Note that the choice of $\eps$ made at the beginning of this paragraph ensures that the interior of $\mathscr{C}$ contains
$\Theta_{\eps,\alpha}^{+}$ (Figure~\ref{fig:position spectrum A L limfluct}).

The argument is the following: by construction of $\mathscr{C}$, $0$ is the only eigenvalue (with multiplicity $2$) of the operator $A$ lying in the interior of $\mathscr{C}$. A principle of continuity of eigenvalues shows that, while adding a small enough perturbation $B$ to $A$, the interior of $\mathscr{C}$ still contains either one eigenvalue with algebraic multiplicity $2$ or two eigenvalues with multiplicity $1$; those perturbed eigenvalues remain close but are \emph{a priori} not equal to the initial eigenvalue $0$ (see Figure~\ref{fig:position spectrum A L limfluct}).
\begin{figure}[!ht]
\centering
\subfloat[Position of the spectrum for the self-adjoint operator $A$]{\includegraphics[width=7cm]{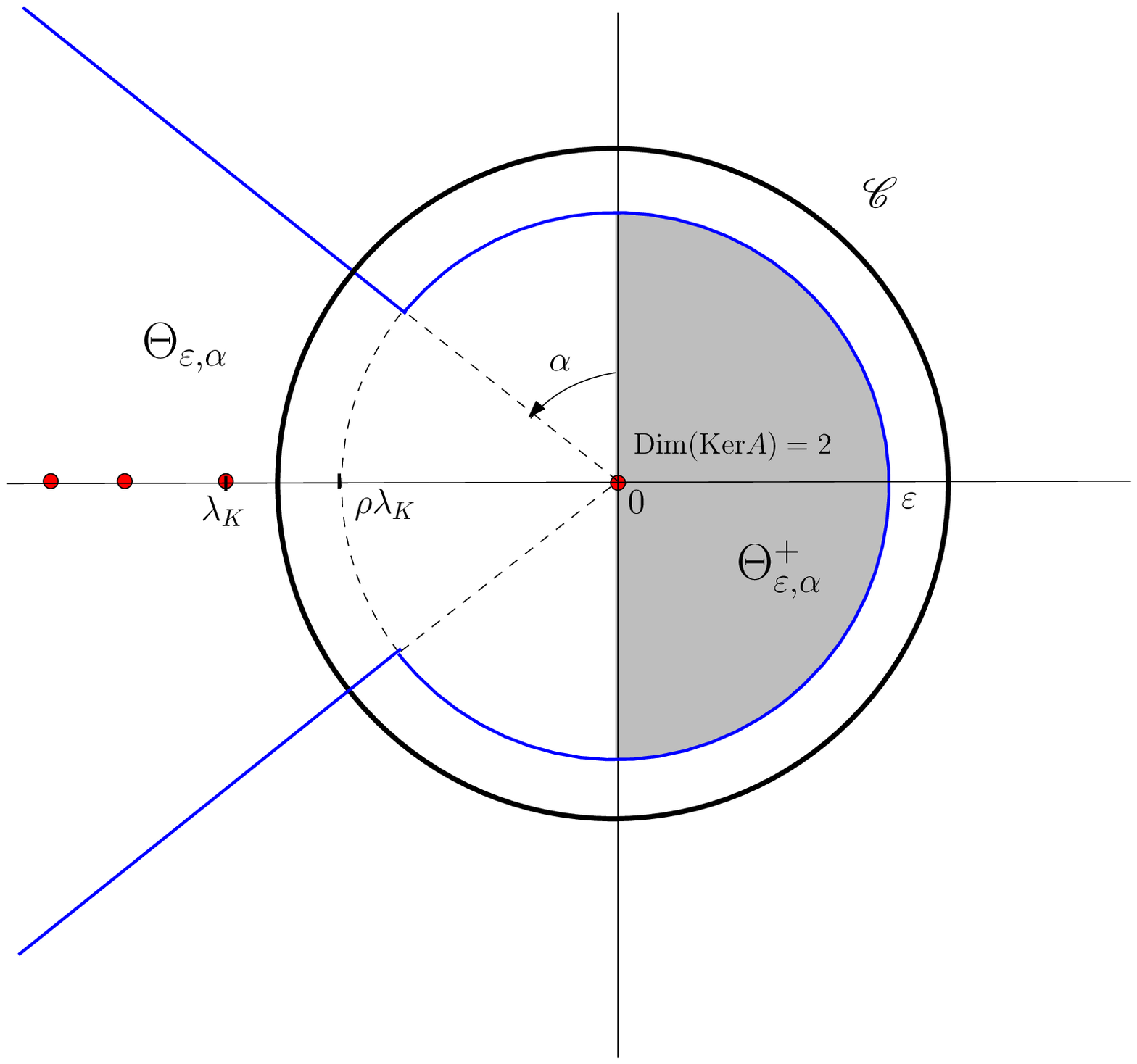}
\label{subfig:coinA limfluct}}
\quad\subfloat[Possible position of the spectrum for the operator $L$]{\includegraphics[width=7cm]{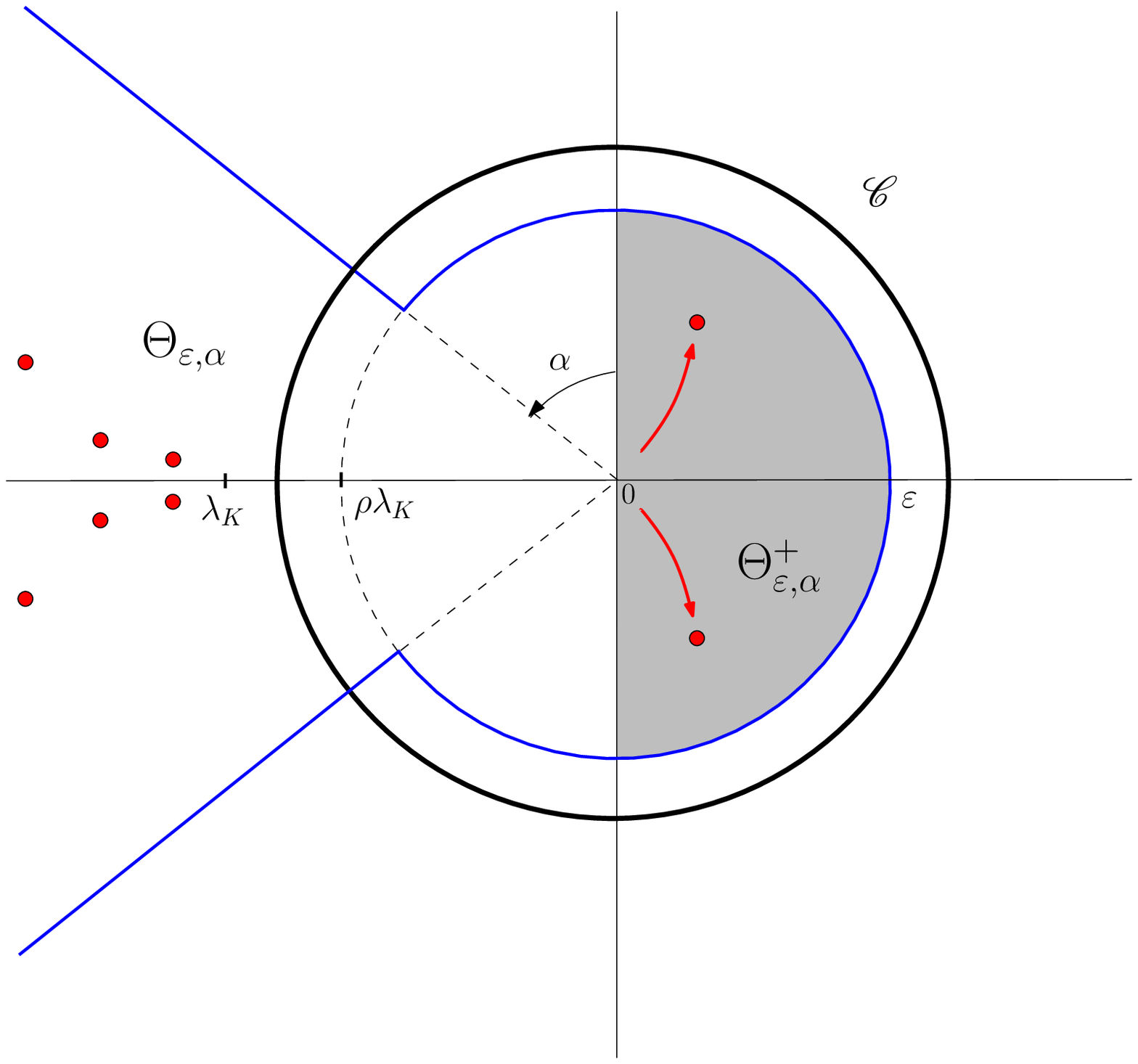}
\label{subfig:coinL limfluct}}
\caption{The domains $\Theta_{\eps, \alpha}$ and $\Theta_{\eps, \alpha}^+$ (in light grey). Note that the two dimensional eigenvalue $0$ for the operator $A$ (\ref{subfig:coinA limfluct}) may split in two single eigenvalues for the perturbed operator $L$ (\ref{subfig:coinL limfluct}). These eigenvalues are the only ones within the circle $\mathscr{C}$. But since we already know that the eigenspace in $0$ of $L$ is of dimension $2$, $0$ is still a double eigenvalue for $L$, by uniqueness.}%
\label{fig:position spectrum A L limfluct}%
\end{figure}

But we already know that for the perturbed operator $L=A+B$, $0$ is always an eigenvalue (since $L\partial_\tta q=0$ and $L p=\partial_\tta q$,
recall
Th.~\ref{th:existence
jrodan block intro limfluct}). Therefore, the algebraic multiplicity of $0$ for the operator $L$ is \emph{at least} $2$. By uniqueness, one can
conclude that $0$ is the \emph{only} element of the spectrum of $L$ within $\mathscr{C}$, and is an eigenvalue with algebraic multiplicity
\emph{exactly} $2$. In particular, there is no element of the spectrum in the positive part of the complex plane.

\medskip

In order to make this argument precise, we need to quantify the appropriate size of the perturbation $B$, by explicit estimates on
the resolvent $R(\lambda, A)$ on the circle $\mathscr{C}$:

\begin{lem}
 \label{lem:estimRGamma}
There exists some explicit constant $c_\mathscr{C}(K)$ only dependent on $K$, such that for all $\lambda\in\mathscr{C}$,
\begin{align}
 \NH{R(\lambda, A)} &\leq c_\mathscr{C}(K),\label{eq:estimRcercle limfluct}\\
\NH{AR(\lambda, A)} &\leq 1+ \left( \frac{\rho+1}{2} \right)\cdot c_\mathscr{C}(K).\label{eq:estimARcercle limfluct}
\end{align}
One can choose $c_\mathscr{C}(K)$ as $c_\mathscr{C}(K)=\frac{1}{\gapA}\max\left(\frac{2}{\rho+1}, \frac{2}{1-\rho}\right)$.
\end{lem}
\begin{proof}[Proof of Lemma~\ref{lem:estimRGamma}]
Applying the spectral theorem (see~\cite[Th. 3, p.1192]{Dunford1988}) to the essentially self-adjoint operator $A$, there exists a
spectral measure $E$ vanishing on the complementary of the spectrum of $A$ such that $A=\intR \lambda \dd E(\lambda)$. In that
extent, one has for any $\zeta\in\mathscr{C}$
\begin{equation}
 R(\zeta, A)\, =\, \intR \frac{\dd E(\lambda)}{\lambda -\zeta}\, .
\end{equation}
In particular, for $\zeta\in\mathscr{C}$
\begin{equation}
 \NH{R(\zeta, A)} \leq \sup_{\lambda\in \sigma(A)} \frac{1}{|\lambda-\zeta|}\leq
\frac{\max\left(\frac{2}{\rho+1}, \frac{2}{1-\rho}\right)}{\gapA}\, .
\end{equation}
The estimation \eqref{eq:estimARcercle limfluct} is straightforward.
\end{proof}

We are now in position to apply our argument of local continuity of eigenvalues: following~\cite[Th III-6.17, p.178]{Kato1995}, there exists a
decomposition of the operator $A$ according to $\Hsmu=F_0 \oplus \Fneg$ (in the sense that $AF_0\subset F_0$,
$A \Fneg\subset \Fneg$ and $P_0 \mD \subset \mD$, where $P_0$ is the projection on $F_0$ along $\Fneg$) in such a way that $A$
restricted to $F_0$ has spectrum $\{0\}$ and $A$ restricted to $\Fneg$ has spectrum $\sigma(A)\smallsetminus\{0\}\subseteq
\ens{\lambda\in\bC}{\Re(\lambda)< 0}$. Let us note that the dimension of $F_0$ is \emph{exactly} $2$, since the characteristic space of $A$ for
the eigenvalue $0$ is reduced to its kernel which is of dimension $2$ (see Prop.~\ref{prop:Aselfadj}).

Then, applying~\cite[Th. IV-3.18, p.214]{Kato1995}, and using Proposition~\ref{prop:Bperturb},
we find that if one chooses $\om_2>0$, such that
\begin{equation}
\label{eq:condperturbGamma}
 \sup_{\lamb\in\mathscr{C}} \left( a(\om_2) \NH{R(\lamb, A)} + b(\om_2)\NH{AR(\lamb, A)}\right)<1,
\end{equation}
then for all $0<\om_0<\om_2$, the perturbed operator $L$ is likewise decomposed according to $\Hsmu=G_0\oplus \Gneg$, in such a way that
$\dim(F_0)=\dim(G_0)=2$, and that the spectrum of $L$ is again separated in two parts by $\mathscr{C}$. But thanks to Theorem
\ref{th:existence jrodan block intro limfluct}, we already know that the characteristic space of the perturbed operator $L$ according to the
eigenvalue $0$ is \emph{at least} of dimension $2$ (since $L\partial_\tta q=0$ and $Lp=\partial_\tta q$). We can conclude, that for such an
$0<\om_0<\om_2$, $0$ is the
only
eigenvalue in $\mathscr{C}$ and that $\dim(G_0)$ is \emph{exactly} $2$.

Applying Lemma~\ref{lem:estimRGamma}, we see that condition \eqref{eq:condperturbGamma} is satisfied if we choose $\om_2>0$
so that:
\begin{equation}
 \label{eq:condomega99}
a(\om_2) c_\mathscr{C}(K) + b(\om_2)\left( 1 + \left( \frac{\rho+1}{2} \right) c_{\mathscr{C}}(K) \right)<1.
\end{equation}
In that case, the spectrum of $L$ is contained in $\ens{z\in\bC}{\Re(z)\leq 0}$. Theorem~\ref{theo:spectreLq} is proved.
\section{Non self-averaging phenomenon for the fluctuation process (Proof of Theorem~\ref{theo:long time evolution eta limfluct})}
\label{sec:limit behavior etat limfluct}
The purpose of this section is to prove Theorem~\ref{theo:long time evolution eta limfluct}, that is the linear asymptotics \eqref{eq:linear behavior eta} for the SPDE \eqref{eq:SPDE fluctuations limfluct}.

In our framework, (recall that $\mu=\frac 12(\delta_{-\om_0} + \delta_{\om_0})$), the solution $\eta$ of evolution \eqref{eq:SPDE fluctuations
limfluct} in $\mS'(\Sun\times\bR)$ acts on test functions $\vphi$ of the form $\vphi=(\vphi(\cdot, +\om_0), \vphi(\cdot, -\om_0))$. In
particular, one can understand $\eta$ as an element of $\Hsmu$ by identifying $\eta$ with $(\eta_{\om_0}, \eta_{-\om_0})$, where, for any
smooth function $\psi: \Sun\to \bR$, $\eta_{\om_0}(\psi):= \eta(\psi, 0)$ and $\eta_{-\om_0}(\psi):= \eta(0, \psi)$. Defining analogously
$W_{\pm\om_0}$ for the Wiener process $W$ in \eqref{eq:W process intro limfluct}, the object of interest is then
\begin{equation}
 \label{eq:SPDE fluctuations plusmoins limfluct}
\forall t>0,\ \eta_t = \eta_0 + \int_{0}^{t} L\eta \dd s + W_t.
\end{equation}
\subsection{The noise $W$ as a cylindrical Brownian Motion}
\label{subsec:W cylindrical BM limfluct}
We first focus on the regularity properties of the noise $W$: in the stationary case ($q_t=q_t|_{t=0}=q$ for all $t>0$) the covariance defined
in
\eqref{eq:W process intro limfluct} becomes, for any regular functions on $\Sun\times\bR$, $\vphi_1$ and $\vphi_2$, $s,t>0$:
\begin{equation}
 \label{eq:new covariance W SPDE limfluct}
\bE\left( W_s(\vphi_1) W_t(\vphi_2)\right)= (s\wedge t) \intR\intS \partial_\tta\vphi_{1}\partial_\tta\vphi_{2} q \dd\lambda\dd\mu.
\end{equation}
Consequently, it is easily seen that $(W_{+\om_0, t}, W_{-\om_0, t})$ is a couple of two independent Gaussian processes with covariance (where
$\psi_1, \psi_2:\Sun\to\bR$):
\begin{equation}
 \label{eq:covariance Wtplus limfluct}
\forall \om\in\{+\om_0, -\om_0\},\ \bE\left( W_{\om, s}(\psi_1) W_{\om, t}(\psi_2) \right)= \frac12 (s\wedge t)\intS
\partial_\tta\psi_1\partial_\tta\psi_2 q(\cdot, \om).
\end{equation}
In what follows, we will denote by $\Hz$ the closed subspace of $\Hsmu$ consisting of elements of $\Hsmu$ with zero mean-value; in particular
the norm
$\NH{\cdot}$ defined in \eqref{eq:norm Hmoins 1 weight k mu limfluct} coincides on $\Hz$ with:
\begin{equation}
 \label{eq:norm Hz limfluct}
\forall h\in\Hz,\ \NH{h}=\left(2\pi\intR\intS\mH^2\right)^{\frac12},
\end{equation}
where we recall that $\mH$ is the primitive of $h$ such that $\intS \mH=0$.
Following~\cite[p. 96]{DaPrato1992}, $W$ has the same law as a $Q$-Wiener process in
the Hilbert space $\Hz$, for an appropriate bounded symmetric operator $Q$ on $\Hz$: indeed, if one denotes by $X$ a
$Q$-Wiener process on $\Hz$, with the following definition of $Q$
\begin{equation}
 \label{eq:Q limfluct}
\forall h\in\Hz,\ Qh:= \partial_\tta\left( q \mH \right),
\end{equation}
then one readily verifies that the Gaussian process $(W(\vphi))_{\vphi}$ has the same law as
$(X(\vphi))_{\vphi}:=\left(\croH{X_t}{\svect{\partial_\tta^2\vphi}{0}}, \croH{X_t}{\svect{0}{\partial_\tta^2\vphi}}\right)_{\vphi}$.

The fact that this supplementary weight $q$ in \eqref{eq:covariance Wtplus limfluct} entails some technical complications, but one really has to consider the operator $Q$ defined in \eqref{eq:Q limfluct} only as a perturbation of the case $Q=I$.

\subsection{Existence and uniqueness of a weak solution to the fluctuation equation}
We now turn to the existence and uniqueness of a weak solution of \eqref{eq:SPDE fluctuations plusmoins limfluct}. We recall that any
$\Hsmu$-valued
predictable process $\eta_t$, $t\in[0, T]$ is a \emph{weak solution} of \eqref{eq:SPDE fluctuations plusmoins limfluct} if the trajectories of
$\eta$ are almost-surely Bochner-integrable and if for all $\vphi\in\mD(L^\ast)$ and for all $t\in[0, T]$
\begin{equation}
 \label{eq:weak solution SPDE limfluct}
\eta_t(\vphi) = \eta_0(\vphi) + \int_0^t \eta_s(L^\ast \vphi) \dd s + W_t(\vphi).
\end{equation}
\begin{prop}
\label{prop:unique weak solution SPDE limfluct}
 Equation has a unique weak solution in $\Hsmu$, given by the mild formulation
\begin{equation}
 \label{eq:mild formulation SPDE limfluct}
\eta_t = T_L(t)\eta_0 + \int_0^t T_L(t-s)\dd W_s,\quad t\in[0, T].
\end{equation}
\end{prop}
To prove Proposition~\ref{prop:unique weak solution SPDE limfluct}, one needs to define properly the stochastic convolution $W_L(t):=\int_0^t
T_L(t-s)\dd W_s$. In this purpose, let use prove firstly that the inverse of $A$ is of class trace:

\begin{lem}
 \label{lem:A of trace class limfluct}
The operator $A^{-1}$ is of class trace in $\Hsmu$. Equivalently, if $(\lambda_n^{(A)})_{n\geq 1}$ is the sequence of eigenvalues of the
self-adjoint operator $A$, one has
\begin{equation}
 \label{eq:A of trace class limfluct}
\Som{n}{1}{\infty}{\frac{1}{\lambda_n^{(A)}}}<\infty.
\end{equation}
\end{lem}
\begin{proof}[Proof of Lemma~\ref{lem:A of trace class limfluct}]
Since $A$ and $\tA=M\circ A\circ M^{-1}$ (recall \eqref{eq:defAselfadj}) are conjugate, it suffices to prove \eqref{eq:A of trace class
limfluct} when $A$ is replaced by $\tA$.
The idea of the proof is that identity \eqref{eq:A of trace class limfluct} is true when $\tA=(\tA_1, \tA_2)$ is replaced by $(-\Delta,
-\Delta)$
and that $\tA$ is only a relatively-bounded perturbation of this case. More precisely, the proof relies on the following MinMax Principle:
\begin{prop}[{\cite[p. 1543]{Dunford1988b}}]
 \label{prop:minimax principle}
Let $(F, \mD(F))$ a self-adjoint linear operator on a separable Hilbert space $\mH$ such that $F$ is positive, with compact resolvent. We
denote
by $\mS^n$ the family of $n$-dimensional subspace of $\mH$, and for $n\geq 1$ we let $\lambda_n$ the number defined as follows
\begin{equation}
\label{eq:minimax limfluct}
 \lambda_n:= \sup_{G\in\mS^n} \inf_{u\in(G\cap\mD(F))\smallsetminus\{0\}} \frac{\cro{u}{Fu}_{\mH}}{\cro{u}{u}_{\mH}}.
\end{equation}
Then there exists a complete orthonormal system $(\psi_n)_{n\geq 1}$ such that 
\[F\psi_n = \lambda_n \psi_n, \quad n\geq1.\]
In other words, the sequence $(\lambda_n)_{n\geq 1}$  is the non-decreasing enumeration of the eigenvalues of $F$, each repeated a number of
times equal to its multiplicity. Moreover, the $\sup$ in \eqref{eq:minimax limfluct} is attained for $G$ equal to the span of $\{\psi_1,
\dots, \psi_n\}$.
\end{prop}
Let us apply Proposition~\ref{prop:minimax principle} to $F=-\Delta$ with domain
\begin{equation}
 \label{eq:domain laplacian}
\mD(-\Delta):=\ens{u}{u\in\mC^2(\Sun),\ \intS u(\tta)\dd\tta =0}\, ,
\end{equation}
in $\bH^{-1}$ (recall the definition of $\left(\bH^{-1}, \cro{\cdot}{\cdot}_{-1}\right)$ in Remark~\ref{rem:weighted spaces no disorder
limfluct}) and let
us denote by $\mE_0(u,v):= \cro{u}{-\Delta v}_{-1} = 2\pi\intS uv$ the Dirichlet form associated to $-\Delta$. Note
that $\mE_0$ is well defined on $\Ld\supset \mD(-\Delta)$. Then, denoting by $(\lambda_n^{(-\Delta)})_{n\geq1}$ the sequence of eigenvalues
associated to $-\Delta$ in $\bH^{-1}$:
\[\lambda_n^{(-\Delta)}=\sup_{G\in\mS^n} \inf_{u\in(G\cap\mD(-\Delta))\smallsetminus\{0\}} \frac{\mE_0(u,u)}{\cro{u}{u}_{-1}}.\]
Since the supremum is attained for $G=\{\psi_1, \dots, \psi_n\}\subseteq \Ld$, one has in fact:
\[\lambda_n^{(-\Delta)}=\sup_{G\in\mS^n} \inf_{u\in(G\cap \Ld)\smallsetminus\{0\}} \frac{\mE_0(u,u)}{\cro{u}{u}_{-1}}.\]
Secondly, note that one does not change the result by considering $1-A_1$ instead of $-A_1$. Hence, if ones denotes by $\mE_1(u,v):=
\cro{u}{(1-A_1)v}_{-1, q_0}$ the Dirichlet form associated to $1-A_1$, one deduces from~\cite[Eq.(2.47)]{BGP} that $\mE_1$ is well defined on
$\Ld$ and that it is equivalent to $\mE_0$: there exists a constant $C>0$ such that
\[\forall u\in \Ld, \quad \frac1C \mE_0(u,u) \leq \mE_1(u,u) \leq C \mE_0(u,u).\]
Then, using again Proposition~\ref{prop:minimax principle},
\[\lambda_n^{(1-A_1)}= \sup_{G\in\mS^n} \inf_{u\in(G\cap \Ld)\smallsetminus\{0\}} \frac{\mE_1(u,u)}{\cro{u}{u}_{-1, q_0}}.\]
Since the norms $\Vert \cdot\Vert_{-1}$ and $\Vert \cdot\Vert_{-1, q_0}$ are equivalent, one directly sees that there exist
constants $c,C>0$
such that, for all $n\geq 1$
\begin{equation}
\label{eq:bounds for lambda A1}
 c\lambda_n^{(-\Delta)}\leq \lambda_n^{(1-A_1)} \leq C \lambda_n^{(-\Delta)}.
\end{equation}
One can prove similar bounds for $A_2$ in the Hilbert space $\HsPhi$ in the same way: first notice that any eigenvector which corresponds to a
non-zero eigenvalue of $A_2$ is necessarily with zero mean-value, so that it suffices to work on the domain $\{v\in \Ld,\ \intS v=0\}$.
It is then easy to deduce from \eqref{eq:estim bellow Dir form A2 limfluct} that both
Dirichlet forms $\mE_0$ and $\mE_{2}$ (recall definition \eqref{eq:dirichlet form for A2}) are equivalent on the subspace of $\Ld$ with
zero mean-value. Using Proposition~\ref{prop:minimax principle}, one easily obtains similar bounds as \eqref{eq:bounds for lambda
A1} for $A_2$ and \eqref{eq:A of trace class limfluct} follows.
\end{proof}

\medskip

Following the lines of~\cite{DaPrato1992}, we deduce that the linear operator $\int_0^t T_L(s) Q T_L(s)^\ast\dd s$ is of class trace: indeed, it is easy to see from \eqref{eq:operator B L limfluct} that $B$ satisfies the assumption (5.59) in~\cite[p.145]{DaPrato1992}, namely, $B$ is a continuous linear operator from $\Ldmu$ into $\Hsmu$, and there exists a constant $c>0$ such that for all $h\in\mD$, $\croH{Bh}{h} \leq
c\NH{h}^2$. Since Lemma~\ref{lem:A of trace class limfluct} is true, the assumptions of~\cite[Prop. 5.25, p.145]{DaPrato1992} are satisfied, so that the operator $\int_0^t T_L(s) Q T_L(s)^\ast\dd s$ is of class trace. Then an application of~\cite[Th. 5.2]{DaPrato1992} shows that the stochastic convolution $W_L(\cdot)$ is well defined as a predictable process in $\Hsmu$. The assumptions of~\cite[Th. 5.4]{DaPrato1992} concerning the existence and uniqueness a weak solution of \eqref{eq:SPDE fluctuations plusmoins limfluct} are satisfied and Proposition~\ref{prop:unique weak solution SPDE limfluct} is
proved.

\subsection{Linear asymptotic behavior of the fluctuation process}
We are in position to prove the main statement of Theorem~\ref{theo:long time evolution eta limfluct}, that is the asymptotic behavior \eqref{eq:linear behavior eta} of the mild solution \eqref{eq:mild formulation SPDE limfluct}. We place ourselves under the hypothesis of Theorems~\ref{th:existence jrodan block intro limfluct} and \ref{theo:spectreLq}.

First note that the continuous linear form $\ell_{\partial_\tta q}(\cdot)$  \eqref{eq:matrix representation L H} on $\Hsmu$ can be represented, by Riesz representation theorem, as a scalar product w.r.t. some vector $\zeta\in \Hsmu$: \begin{equation}
\label{eq:lqprime zeta}
\ell_{\partial_\tta q}(\cdot)=\croH{\zeta}{\cdot}.
\end{equation}
The convergence \eqref{eq:linear behavior eta} is a consequence of Remark~\ref{rem:expression of lp} and the following two propositions:
\begin{prop}
 \label{prop:linear behavior Pzero eta}
The stochastic convolution $W_L(t)$ satisfies the following linear behavior, as $t\to+\infty$: $\frac{W_L(t)}{t}\to0$, where the convergence is
in law.
\end{prop}
\begin{prop}
 \label{prop:linear behavior Pstrict eta}
For every initial condition $\eta_0$, $\frac{T_L(t)\eta_0}{t}$ converges, as $t\to+\infty$, to $\ell_p(\eta_0)\partial_\tta q$.
\end{prop}
Before proving these results, let us show how the speed $v(\om)$ in Theorem~\ref{theo:long time evolution eta limfluct} is computed. The above results give that for fixed $\om$, the process $ \frac{\eta_{t}^{\om}}{t}$ converges in law, as $t\to\infty$ to $\ell_{p}(\eta_{0}^{\om}\partial_{\tta}q = \frac{\int_{\bS}C_{+}(\om)}{\int_{\bS} p_{+}} \partial_{\tta}q$. Using \eqref{eq:covariance C}, an easy computation shows that $v(\cdot)$ is Gaussian with variance \eqref{eq:variance speed intro}. Let us now prove these two propositions:
\begin{proof}[Proof of Proposition~\ref{prop:linear behavior Pzero eta}]
Recall that $W$ is a $Q$-Wiener process in $\Hz$, which can be decomposed into $\Hz=\Span(\partial_\tta q)\oplus \Gneg$. Note also that the restriction on $\Hz$ of the projection $\Pzero$ defined on $\Hsmu$ by \eqref{eq:ell projection P0} coincides with $\ell_{\partial_\tta q}(\cdot)\partial_\tta q$. With a small abuse of notations, we will use the same notation $\Pzero$ for this restriction on $\Hz$. Let us decompose the stochastic convolution into $W_L(t)=\int_0^t T_L(t-s) \Pzero \dd W_s + \int_0^t T_L(t-s) \Pneg \dd W_s$, and
treat the two terms separately: 

\noindent For the first term $\frac{1}{t}\int_0^t T_L(t-s) \Pzero \dd W_s$, one has, using that $T_L(u) \partial_\tta
q=\partial_\tta q$ for all $u>0$
\begin{align}
 \frac{1}{t}\int_0^t T_L(t-s) \Pzero \dd W_s &= \frac{\partial_\tta q}{t}\int_0^t T_L(t-s) \ell_{\partial_\tta q} \dd W_s,\\
&=\frac{\partial_\tta q}{t}\ell_{\partial_\tta q}W_t= \frac{\partial_\tta q}{t}\croH{\zeta}{W_t},\ \text{(recall \eqref{eq:lqprime zeta})}.
\end{align}
Thanks to the $Q$-Wiener structure of $W$ (see \S~\ref{subsec:W cylindrical BM limfluct}), one has
\begin{equation}
 \bE\left( \NH{\frac{1}{t}\int_0^t T_L(t-s) \Pzero \dd W_s}^2 \right)= \frac{\NH{\partial_\tta q}^2}{t^2}\bE\left( |\croH{\zeta}{W_t}|^2
\right)=
\frac{\NH{\partial_\tta q}^2}{t}\croH{Q\zeta}{\zeta},
\end{equation}
which converges to $0$ as $t\to+\infty$.

\noindent For the second term $\int_0^t T_L(t-s) \Pneg \dd W_s$, it is easy to see that it is the unique weak solution in $\Hz$ of
\begin{equation}
 \label{eq:SPDE stochastic convolution}
w_t = \int_0^t Lw_s \dd s + \Pneg W_t.
\end{equation}
Let us decompose evolution \eqref{eq:SPDE stochastic convolution} along this decomposition $\Hz=\Span(\partial_\tta q)\oplus\Gneg$: writing
$w_t= \Pzero w_t +
\Pneg w_t:= y_t + z_t$, one has:

\begin{equation}
 \left\{\begin{array}{ccl}
         z_t&=&\int_0^t \Pzero Ly_s\dd s,\\
y_t&=&\int_0^t \Pneg L\Pneg y_s\dd s + \Pneg W_t.
        \end{array}
\right.
\end{equation}
Since the operator $\Pneg L\Pneg$ has its spectrum in the negative part of the complex plane with a strictly positive spectral gap $\gapL$ and
generates an semigroup of operators, it is immediate to see from the covariance estimates of stochastic convolutions (see
\cite[Th. 5.2, p.119]{DaPrato1992}) that there exist some $t_0>0$ and a constant $c>0$ such that for all $t\geq t_0$
\begin{equation}
 \label{eq:estim yt covariance limfluct}
\bE\left( \NH{y_t} \right)^2\leq \bE\left( \NH{y_t}^2 \right)\leq ce^{-\frac{\gapL}{2} t}
\end{equation}
Consequently, one has
\begin{align}
 \bE\left( \NH{z_t} \right)&\leq \int_0^t \bE\left( \NH{\Pzero Ly_s} \right)\dd s= \NH{\partial_\tta q} \int_0^t
\bE\left(|\ell_{\partial_\tta q}(Ly_s)|\right)\dd
s,\nonumber\\
&= \int_0^t \bE\left(  |\croH{\zeta}{Ly_s}|\NH{\partial_\tta q} \right)\dd s,\ \text{(recall \eqref{eq:lqprime zeta})}\nonumber\\
&\leq \NH{\partial_\tta q} \NH{L^\ast \zeta}\int_0^t \bE\left(\NH{y_s}\right)\dd s,\nonumber\\
&=\NH{\partial_\tta q} \NH{L^\ast \zeta}\left(\int_0^{t_0} \bE\left(\NH{y_s}\right)\dd s + \int_{t_0}^t
\bE\left(\NH{y_s}\right)\dd s\right).\label{eq:last estim zt}
\end{align}
It is immediate from \eqref{eq:estim yt covariance limfluct} and \eqref{eq:last estim zt} that the following convergence holds:
\begin{equation}
 \label{eq:convergence zt}
\frac{\bE\left( \NH{z_t} \right)}{t} \to_{t\to+\infty} 0.
\end{equation}
Putting together \eqref{eq:estim yt covariance limfluct} and \eqref{eq:convergence zt}, Proposition~\ref{prop:linear behavior Pzero eta} is
proved.
\end{proof}

\begin{proof}[Proof of Proposition~\ref{prop:linear behavior Pstrict eta}]
Let us fix an initial condition $\eta_0\in \Hsmu$. 
Then $X(t):=T_L(t)\eta_0$ is the unique solution in $\Hsmu$ of
\begin{equation}
 \label{eq:equation for TLt etazero limfluct}
X(t)=\eta_0 +\int_0^t LX_s \dd s.
\end{equation}
Decompose $X(t)$ along the direct sum $G_0\oplus\Gneg$, that is $X(t)=\alpha(t) \partial_\tta q+\beta(t) p+ Y(t)$, with $Y(t)\in\Gneg$. Then,
projecting on
$\partial_\tta q$, $p$ and $\Gneg$ respectively (see \eqref{eq:matrix representation L H}), one obtains that \eqref{eq:equation for TLt etazero
limfluct} is
equivalent to
\begin{equation}
 \label{eq:equation for TLt etazero equivalent form}
\forall t>0,\quad\left\{\begin{array}{ccl}
 \alpha(t) &=& \ell_{\partial_\tta q}(\eta_0) + \int_0^t \left(\beta(s) + \ell_{\partial_\tta q}(L\Pneg Y(s))\right)\dd s,\\
\beta(t) &=& \ell_p(\eta_0),\\
Y(t)&=& \Pneg\eta_0 + \int_0^t \Pneg L\Pneg Y(s)\dd s.
\end{array}\right.
\end{equation}
Then, since $T_{\Pneg L\Pneg}(t)$ is a semigroup of contraction whose infinitesimal generator has a strictly positive spectral
gap $\gapL$, there exists a constant $c>0$ such that $Y(t)=T_{\Pneg L\Pneg}(t)\Pneg\eta_0$ and $\NH{Y(t)}\leq
ce^{-\frac{\gapL}{2}
t}$ (in particular, $\frac1t\NH{Y(t)}\rightarrow_{t\to+\infty}0$). Then, using again \eqref{eq:lqprime zeta},
\begin{align}
 \frac{\alpha(t)}{t} &=\frac{\ell_{\partial_\tta q}(\eta_0)}{t} + \ell_p(\eta_0) + \frac1t\int_0^t \ell_{\partial_\tta q}(L\Pneg Y(s))\dd s,\\
&=\frac{\ell_{\partial_\tta q}(\eta_0)}{t} + \ell_p(\eta_0) + \frac1t\int_0^t \croH{\Pneg^\ast L^\ast \zeta}{Y(s)}\dd s.
\end{align}
Using the previous exponential bound for $Y(s)$, it is easy to see that $\frac{\alpha(t)}{t}$ converges to $\ell_p(\eta_0)$ as $t\to+\infty$.
The
result of Proposition~\ref{prop:linear behavior Pstrict eta} follows.
\end{proof}

\begin{appendix}
\section{Gelfand-triple construction}
\label{app:gelfand}
The construction of the weighted Sobolev spaces defined in \S~\ref{subsec:weighted sobo limfluct} and used throughout the
paper is based on the usual notion of Gelfand-triple (or rigged-Hilbert spaces) that we make precise here. We refer to~\cite[\S~2.2]{BGP} or \cite[p.81]{MR697382} for precise definitions.

Namely, one can understand the closure of $\ens{h\in\mC^1(\Sun)}{\intS h=0}$ \wrt the norm $\Vert h \Vert_{-1, k}$ defined in \eqref{eq:norm
Hmoins 1 weight k limfluct} as the dual space $V'$ of the space $V$ closure of
$\ens{h\in\mC^1(\Sun)}{\intS h=0}$ with respect to the norm \[\Vert{h}\Vert_V:= \left( \intS h'(\tta)^2 k(\tta)\dd\tta\right)^\frac12.\] The
pivot space is the usual
$\bL^2(\lambda)$, endowed with the Hilbert norm \[\Vert h\Vert_{\bL^2}:=\left( \intS h(\tta)^2\dd\tta\right)^\frac12.\] One easily sees that the inclusion $V\subseteq \bL^2(\lambda)$ is dense.
Consequently, one can define $T:\bL^2(\lambda)\rightarrow V'$ by setting $Th(v)= \intS h(\tta) v(\tta)\dd\tta$.
One can prove that $T$ continuously injects $\bL^2(\lambda)$ into $V'$ and that $T(\bL^2(\lambda))$ is dense into $V'$ so that one can
identify $h\in \bL^2(\lambda)$ with $Th\in V'$. Then for $h\in \bL^2(\lambda)$,
\begin{equation}
\label{eq:Gelfand triple}
 \Vert h\Vert_{V'} = \Vert Th\Vert_{V'} = \sup_{v\in V} \frac{\intS \mH h'}{\Vert v\Vert_{V}}= \sqrt{\intS \frac{\mH^2}{k}}\, ,
\end{equation}
where we used in \eqref{eq:Gelfand triple} Cauchy-Schwarz inequality for the lower bound and chose $v':=\frac{\mH}{k}$ for the upper bound.
This enables to identify $\bH^{-1}_{k}$ with $V'$.
\end{appendix}

\section*{Acknowledgements} This is a part of my PhD thesis. I would like to thank my PhD supervisors Giambattista Giacomin and Lorenzo Zambotti for introducing this subject, for their useful advice, and for their encouragement. I am also indebted to Christophe Poquet for useful discussions. 

I would like to thank the referee for useful comments and suggestions.

\section*{References}
\def\cprime{$'$}

\end{document}